\documentclass{amsart}
\usepackage[margin=2.5cm]{geometry}
\usepackage[svgnames]{xcolor}
\usepackage{enumitem}
\usepackage{amsmath,amssymb}
\usepackage{mathtools}
\usepackage{framed}
\usepackage{tikz}
\usepackage{dsfont}
\usetikzlibrary{shapes.geometric}
\usetikzlibrary{arrows,arrows.meta}	
\usepackage[latin1]{inputenc}
\usepackage[english]{babel}
\usepackage{multimedia}
\usepackage{multicol}
\usepackage[T1]{fontenc}
\usepackage{times}
\usepackage{graphicx}
\usepackage{tabmac}
\usepackage{pb-diagram}

\newcommand{\tcercle}[1]{\ensuremath{\setlength{\unitlength}{1ex}\begin{picture}(2.8,2.8)\put(1.4,1.4){\circle{2.8}\makebox(-5.6,0){#1}}\end{picture}}}

\newcommand{\ct}[2]{\textrm{ct}({#1,#2})}

\newcommand{\omitt}[1]{}

\newcommand{\Q}{\mathbb Q}

\DeclareMathOperator{\sign}{sign}
\DeclareMathOperator{\Rev}{rev}
\newcommand{\ga}{\alpha}
\newcommand{\gb}{\beta}
\newcommand{\fund}[1]{L_{#1}}
\newcommand{\funda}{\fund{\ga}}
\newcommand{\te}{\theta}

\def\SPar{\text{SPar}}
\def\SParnm{\SPar(n|m)}

\def\defstyle{\bf }

\numberwithin{equation}{section}

\theoremstyle{definition}
  \newtheorem{theorem}{Theorem}[section]
  
  \newtheorem{proposition}[theorem]{Proposition}

  \newtheorem {corollary}[theorem]{Corollary}
  \newtheorem{definition}[theorem]{Definition}
  \newtheorem{example}[theorem]{Example}

\theoremstyle{remark}

\newtheorem*{acknow}{Acknowledgments}

\newsavebox{\smlmat}
\savebox{\smlmat}{$\left(\begin{smallmatrix}1&2&3&4&5&6\\2&4&5&3&6&1\end{smallmatrix}\right)$}
\newsavebox{\smlmatInv}
\savebox{\smlmatInv}{$\left(\begin{smallmatrix}1&2&3&4&5&6\\6&1&4&2&3&5\end{smallmatrix}\right)$}

\title{Hopf algebra structure of symmetric and quasisymmetric functions in superspace}
\author{Susanna Fishel}
\address{School of Mathematical and Statistical Sciences, Arizona State University, P.O. Box 871804, Tempe, AZ 85287-1804, USA} \email{sfishel1@asu.edu}

\author{Luc Lapointe}
\address{Instituto de Matem\'atica y F\'{\i}sica, Universidad de Talca, Casilla 747, Talca,
  Chile} \email{lapointe@inst-mat.utalca.cl}

\author{Mar\'{\i}a Elena Pinto}
\address{Instituto de Matem\'atica y F\'{\i}sica, Universidad de Talca, Casilla 747, Talca,
Chile} \email{mepinto@inst-mat.utalca.cl}


\date{\today}
\begin{document}
\maketitle
\begin{abstract}
  We show that the ring of symmetric functions in superspace is a cocommutative and self-dual Hopf algebra. We provide formulas for the action of the coproduct and the antipode on various bases of that ring.
  We introduce the ring ${\rm sQSym}$ of quasisymmetric functions in superspace and show that it is a Hopf algebra.  We give explicitly the product, coproduct and antipode on the basis of monomial quasisymmetric functions in superspace.  We prove that the Hopf dual of ${\rm sQSym}$, the ring ${\rm sNSym}$ of noncommutative symmetric functions in superspace,  has a multiplicative basis dual to the monomial quasisymmetric functions in superspace.
   \end{abstract}

\section{Introduction}
An extension to superspace of the theory of symmetric functions, originating 
from 
the study of the supersymmetric generalization of the trigonometric Calogero-Moser-Sutherland model,  was developed in \cite{BDLM2,DLM1,DLM2}.  
In this superspace setting, the polynomials
$f(x,\theta)$, where $(x,\theta)=(x_1,\dots,x_N,\theta_1,\dots,\theta_N)$, not only depend on the usual commuting variables
$x_1,\dots,x_N$ but also on the anticommuting variables 
$\theta_1,\dots,\theta_N$ (such that $\theta_i \theta_j=- \theta_j \theta_i, \text{~and~} \theta_i^2=0)$.  Natural generalizations of the
monomial, power-sum, elementary and homogeneous symmetric functions, as well
as of the Schur \cite{BM,JL}, Jack and Macdonald polynomials have been studied.
To illustrate the surprising richness of the theory of symmetric functions in superspace, we could mention that there is even an extension to superspace of the original Macdonald positivity conjecture.

The ring ${\rm QSym}$ of quasisymmetric functions, which can be seen as a refinement of the ring of symmetric functions, was introduced in \cite{Gessel} while its Hopf algebra structure was studied in \cite{E}.
The ring  ${\rm QSym}$ has many applications to symmetric function theory
such as the  elegant expansion of Macdonald polynomials
in terms of fundamental quasisymmetric functions \cite{HHL}.
The  Hopf dual of ${\rm QSym}$, 
the ring of noncommutative symmetric functions ${\rm NSym}$, was defined in \cite{GKLLRT}.

In this article, we undertake to extend to superspace the rich connection between symmetric function theory and quasisymmetric functions.    
As as first step in this direction, the goal of this article is
twofold: (1) extend to superspace
the well-known Hopf algebra structure of the ring of symmetric functions  and (2) introduce the ring of quasisymmetric functions in superspace ${\rm sQSym}$ and understand its Hopf algebra structure as well as its Hopf dual,  the ring
of noncommutative symmetric functions in superspace ${\rm sNSym}$.

To obtain the Hopf algebra structure of the ring ${\mathbf \Lambda}$
of symmetric functions in superspace 
turns out to be relatively straightforward.  We can give explicitly
the coproduct on the basis of power-sum, elementary and homogeneous symmetric functions.  Less trivial to obtain are formulas for the coproduct on the Schur functions in superspace $s_\Lambda$ and $\bar s_\Lambda$, which we derive from Cauchy-type identities.  We then use these formulas to show that the ring of symmetric functions in superspace is a cocommutative and self-dual Hopf algebra.  The action of the antipode $S$ on these various bases is obtained by relating $S$ to a certain well understood involution $\omega$ on ${\mathbf \Lambda}$.

The ring of quasisymmetric functions in
superspace ${\rm sQSym}$ is obtained naturally from  ${\rm QSym}$ by allowing
each variable $x_i$ to be paired with an anticommuting variable $\theta_i$ (which gives rise to the concept of dotted composition).  The coproduct in 
${\rm sQSym}$, the main ingredient needed to obtain its Hopf algebra structure,  is defined as in the non-supersymmetric case by establishing a correspondence between  the splitting of two alphabets and the tensor product ${\rm sQSym} \otimes {\rm sQSym}$.  The product, coproduct and antipode is then given explicitly on the basis of  monomial quasisymmetric functions in superspace, with the formulas being somewhat more complicated than in the usual case
due to the presence of anticommuting variables.
We define two families of fundamental quasisymmetric functions in superspace but, as discussed in Section~\ref{Secfund}, we will relegate their study to a forthcoming article \cite{FLP} given the intricacies of the combinatorics at play.

Finally, we introduce the ring of noncommutative symmetric functions in superspace ${\rm sNSym}$ as the Hopf dual of ${\rm sQSym}$.  Just as in the usual quasisymmetric case, it has a multiplicative basis dual to the monomial quasisymmetric functions in superspace.  We show how  the projection of
${\rm sNSym}$ onto ${\mathbf \Lambda}$ is still compatible with the inclusion of ${\mathbf \Lambda}$ into ${\rm sQSym}$ (see Corollary~\ref{corofinal}).

\section{Hopf algebras}  \label{SecHopf}
We give a brief overview of Hopf algebras based on
\cite{D,GR,LMW}.

In the following, we consider $\mathbb K$ to be a field of characteristic 0 (that we will later always take to be $\mathbb Q$).

An {\defstyle associative algebra} $(\mathcal H,m,u)$ is a  $\mathbb K$-algebra $\mathcal H$
with a $\mathbb K$-linear multiplication (or product) $m: \mathcal H \otimes \mathcal H \to \mathcal H$
and a  $\mathbb K$-linear unit map $u:\mathbb K \to \mathcal H$ such that
\begin{equation}
  m \circ (m\otimes 1)= m \circ (1 \otimes m):
  \mathcal H \otimes \mathcal H \otimes \mathcal H \to \mathcal H\, ,   \qquad m(u(k) \otimes a)=ka \qquad {\rm and} \qquad
  m(a\otimes u(k))=ak
    \end{equation}  
for any $a\in \mathcal H$ and $k\in\mathbb K$.  For simplicity, we
often write the product of $a$ and $b$ as $ab$ instead of $m(a\otimes b)$.

We say that $f: \mathcal H \to \mathcal H'$, where $(\mathcal H',m',u')$ is another associative algebra over $\mathbb K$,  
is an {\defstyle algebra morphism} if
\begin{equation} \label{morphism}
f \circ m = m' \circ (f\otimes f) \qquad {\rm and} \qquad f \circ u= u'
\end{equation}  

A  {\defstyle coassociative algebra} $(\mathcal H,\Delta,\epsilon)$ is a  $\mathbb K$-algebra $\mathcal H$
with a $\mathbb K$-linear comultiplication (or coproduct) $\Delta: \mathcal H \to  \mathcal H \otimes \mathcal H$
and a $\mathbb K$-linear counit $\epsilon:\mathcal H \to \mathbb K$ such that
\begin{equation}
  (\Delta \otimes 1)\circ  \Delta= (1 \otimes \Delta) \circ \Delta :
   \mathcal H \to \mathcal H \otimes \mathcal H \otimes \mathcal H\, , \qquad (\epsilon \otimes 1) \circ \Delta(a) = 1\otimes a \qquad {\rm and} \qquad
  (1 \otimes \epsilon) \circ \Delta(a) = a\otimes 1
    \end{equation}  
for any $a \in \mathcal H$.  We say that $f: \mathcal H \to \mathcal H'$, where $(\mathcal H',\Delta',\epsilon')$ is another coassociative algebra over $\mathbb K$,  is a {\defstyle coalgebra morphism} if
\begin{equation} \label{comorphism}
\Delta' \circ f= (f \otimes f) \circ \Delta\, , \qquad {\rm and } \qquad \epsilon= \epsilon' \circ f 
\end{equation}  

A {\defstyle bialgebra} $(\mathcal H,m,u,\Delta,\epsilon)$ 
is an associative algebra $(\mathcal H,m,u)$ together with a coassociative algebra
$(\mathcal H,\Delta,\epsilon)$ such that either
  $(i)$ $\Delta$ and $\epsilon$
are algebra morphisms or $(ii)$ $m$ and $u$ are coalgebra morphisms.
A bialgebra $\mathcal H$ is said to be {\defstyle graded} if it
has submodules $\mathcal H^0,\mathcal H^1,\dots$
such that
\begin{itemize}
\item $\mathcal H = \bigoplus_{n\geq 0} \mathcal H^n$
\item $\mathcal H^i \mathcal H^{j} \subseteq \mathcal H^{i+j}$
 \item $\Delta(\mathcal H^n) \subseteq    \bigoplus_{i+j=n} \mathcal H^i \otimes \mathcal H^j$ 
\end{itemize}  
If $\mathcal H^0$ has dimension 1 over $\mathbb K$, we say moreover that
$\mathcal H$ is {\defstyle connected}.

Finally, a bialgebra $\mathcal H$ is said to be a {\defstyle Hopf algebra}
if there exists a $\mathbb K$-linear map ({\defstyle antipode}) $S: \mathcal H \to \mathcal H$ such that
\begin{equation} \label{defantipode}
m\circ (S \otimes 1) \circ \Delta = u \circ \epsilon = m\circ (1 \otimes S) \circ \Delta 
\end{equation}

We will need the following theorem.
\begin{theorem} \label{hopftheo}
  Every graded connected bialgebra $\mathcal H$ is a Hopf algebra
  with unique antipode defined recursively by the conditions $S(1)=1$ and 
  \begin{equation} \label{eqantipo}
    S(a) = -\sum_{i=0}^{n-1} S(b_i)  c_{n-i} \qquad {\rm whenever }\qquad
    \Delta(a) = a \otimes 1 + \sum_{i=0}^{n-1} b_i \otimes c_{n-i}
  \end{equation}  
  for $a \in \mathcal H^n$, $n \geq 1$, and $b_i,c_i \in \mathcal H^i$.
  \end{theorem}  

Two Hopf algebras $\mathcal H$ and $\mathcal H'$ are dually paired by a
$\mathbb K$-bilinear map $\langle \cdot, \cdot \rangle: \mathcal H' \otimes \mathcal H \to \mathbb K$
whenever
\begin{align}
  & \langle fg, a \rangle =  \langle f \otimes g , \Delta(a) \rangle\, , \qquad
  \langle f, ab \rangle =  \langle \Delta'(f), a \otimes b  \rangle\, , \nonumber \\ \label{selfdual}
  & \langle 1, a \rangle = \epsilon(a)\, , \qquad 
 \langle f,1 \rangle = \epsilon'(f)\, , \qquad {\rm and} \qquad \langle S'(f), a\rangle = \langle f, S(a) \rangle   
\end{align}
for any $f,g \in \mathcal H'$ and $a,b \in \mathcal H$. In the previous equation, the pairing
$ \langle \cdot, \cdot \rangle: (\mathcal H' \otimes \mathcal H') \otimes (\mathcal H \otimes \mathcal H) \to \mathbb K$
is defined as the composition of maps
\begin{equation} \label{scaltau}
  \mathcal H' \otimes \mathcal H' \otimes \mathcal H \otimes \mathcal H
  \xrightarrow{1\otimes \tau \otimes 1}  \mathcal H' \otimes \mathcal H \otimes \mathcal H' \otimes \mathcal H
  \xrightarrow{ \langle \cdot, \cdot \rangle \otimes   \langle \cdot, \cdot \rangle}  \mathbb K(1\otimes1) \to \mathbb K
\end{equation}
where the last map simply sends $1\otimes 1$ to 1
and where the twist map $\tau: \mathcal H' \otimes \mathcal H \to \mathcal H \otimes \mathcal H'$ will be taken in this article to be such that
\begin{equation} \label{twistmap}
\tau (f \otimes a) = (-1)^{\deg(f) \deg(a)} a \otimes f
\end{equation}
whenever $\mathcal H$ and $\mathcal H'$ are graded bialgebras
and $f$ and $a$ are homogeneous elements\footnote{This is the topologist's twist map usually employed when the algebras come from
  homology or cohomology \cite{GR}. We will see later that in our case the grading will be the fermionic degree.}. Note that this amounts to
\begin{equation}
\langle f \otimes g, a \otimes b \rangle =  (-1)^{\deg(f) \deg(a)} \langle f,a \rangle \langle g, b\rangle
  \end{equation}
A non-degenerate pairing satisfying \eqref{selfdual}
always exists if $\mathcal H$ is graded and each of its homogeneous component is finite dimensional. 
When $\mathcal H$ is dually paired to itself we say that $\mathcal H$ is {\defstyle self-dual}.

We should mention finally that, because we use the twist map \eqref{twistmap},
the antipode defined in \eqref{defantipode} is a signed anti-homomorphism such that $S(1)=1$ and 
$S(ab)= (-1)^{\deg(a) \deg(b)}  S(b)S(a)$ for all homogeneous elements $a,b \in \mathcal H$.

\section{Symmetric functions in superspace}

We now present the main concepts of the
theory of symmetric functions in superspace
\cite{BDLM2,DLM1,DLM2}.
\begin{definition}
  \label{def:superpart}
  A {\defstyle superpartition} $\Lambda \in \SPar$ is a pair of partitions
  $(\Lambda^a; \Lambda^s) = (\Lambda^a , \Lambda^s ) = (\Lambda_1,
  \ldots ,\Lambda_m ; \Lambda_{m+1}, \ldots ,\Lambda_N)$, where $\Lambda^a$
  is a partition with $m$ distinct parts and $\Lambda^s$ is a usual
  partition (possibly including a string of 0's a the end).
\end{definition}

We will sometimes denote superpartitions using {\defstyle dotted}
partitions, where we dot the parts from $\Lambda^a$. For instance,
$(4,3,2;4,4,3,1,1,1)$, $(\dot{4},4,4,\dot{3},3,\dot{2},1,1,1)$, and
$(4,3,2;4^2,3,1^3)$ all denote the same superpartition.

Let $\Lambda$ be a superpartition written as in
Definition~\ref{def:superpart}. The {\defstyle total degree} of $\Lambda$ is
$\sum_{i=1}^N\Lambda_i$ and is written $|\Lambda|$. Its {\defstyle
  fermionic degree (or sector)} is $m$. We say $\Lambda$ is a superpartition of $(n|m)$ if
its total degree is $n$ and its fermionic sector is $m$.  The {\defstyle length} of
$\Lambda$, denoted $\ell(\Lambda)$,
is equal to $\ell(\Lambda^s)+m$, where $m$ is the fermionic degree of
$\Lambda$ and $\ell(\Lambda^s)$ is the usual length of the partition
$\Lambda^s$.
The set of all
superpartitions of $(n|m)$ is denoted $\SParnm$. We also define
$\SPar$ to be the set of all superpartitions.

Superpartitions can be represented by a Ferrers' diagram where the dotted
entries in the corresponding dotted partitions have an extra circle at the end.
For instance, the superpartition $(3,1,0;2,1)$ is represented by 
\begin{equation}
{ \scriptsize        {\tableau[scY]{&&&\bl\tcercle{}\\&\\&\bl\tcercle{}\\ \\
    \bl\tcercle{}}}}
\end{equation}
The {\defstyle conjugate} $\Lambda'$ of the superpartition $\Lambda$ is the superpartition
whose diagram is that of $\Lambda$ reflected through the main diagonal.  Conjugating
the previous diagram gives 
\begin{equation}
{ \scriptsize  {\tableau[scY]{&&&&\bl\tcercle{}\\&&\bl\tcercle{}\\ \\
    \bl\tcercle{}}}}
\end{equation}
which means that the
conjugate of $(3,1,0;2,1)$ is $(4,2,0;1)$.

Before defining the ring of symmetric functions in superspace, we need
to define the analogues of the monomial symmetric functions.

\subsection{Monomial symmetric functions.} 
Let $\SPar_N$ be the set of superpartitions whose length is at most $N$.
The functions
$m_\Lambda(x_1,\dots,x_N;\theta_1,\dots,\theta_N)$, $\Lambda\in\SPar_N$, generalize the
monomial symmetric functions. They are defined by
\begin{equation}
  m_\Lambda(x_1,\dots,x_N;\theta_1,\dots,\theta_N)
  = {\sum_{\sigma \in S_N}}' \theta_{\sigma(1)} \ldots
\theta_{\sigma(m)} x_{\sigma(1)}^{\Lambda_1} \ldots
x_{\sigma(N)}^{\Lambda_N},
\end{equation}
where the prime indicates that the sum is over distinct terms.
\begin{example} If $N=3$, we have
$m_{(2;3,1)} =\te_1 x_1^2 x_2^3 x_3 + \te_1 x_1^2 x_3^3 x_2 + \te_2 x_2^2 x_1^3 x_3 + \te_2 x_2^2 x_3^3 x_1 + \te_3 x_3^2 x_1^3 x_2 + \te_3 x_3^2 x_2^3 x_1  $
\end{example}

\subsection{Ring of symmetric functions in superspace}
It is known that 
$$\{ m_\Lambda(x_1,\dots,x_N;\theta_1,\dots,\theta_N) \}_{\Lambda \in \SPar_N}$$ is a basis of the ring ${\bf \Lambda}_N=\mathbb Q[x_1,\dots,x_N;\theta_1,\dots,\theta_N]^{S_N}$
of symmetric functions in superspace in $N$ variables, where $S_N$ is the symmetric group on $N$ elements.  Note that $S_N$ acts diagonally on the two sets of variables, that is, 
\begin{equation}
  \sigma \, f(x_1,\dots,x_N;\theta_1,\dots,\theta_N)= f(x_{\sigma(1)},\dots,x_{\sigma(N)};\theta_{\sigma(1)},\dots,\theta_{\sigma(N)})
\end{equation}  
for any $\sigma \in S_N$ and any
$f\in {\bf \Lambda}_N$.

The monomial symmetric functions in superspace are stable with respect to
the number of variables.  This allows to consider the number of variables $N$ to be  infinite, or equivalently, to consider $m_\Lambda$ as the inverse limit of the monomial in a finite number of variables  $m_\Lambda(x_1,\dots,x_N;\theta_1,\dots,\theta_N)$.  Given that ${\bf \Lambda}_N$ is bi-graded with respect to the
total degree and the fermionic degree, we can define the ring of symmetric functions in superspace as
\begin{equation}
  {\bf \Lambda} = \bigoplus_{n,m \geq 0} {\bf \Lambda}^{n,m}
\end{equation}
where
\begin{equation}
{\bf \Lambda}^{n,m} = \left \{ \sum_\Lambda c_{\Lambda} m_\Lambda \, | \, c_\Lambda \in \mathbb Q, \,  \Lambda \in \SParnm  \right \}
\end{equation}  
At times, it will also be convenient to use the simple grading with respect to the sum of the total degree and the fermionic degree:
\begin{equation} \label{newgrad}
 {\bf \Lambda}^n=  \bigoplus_{j+k=n} {\bf \Lambda}^{j,k}
\end{equation}

All the other important bases of the ring of symmetric functions have generalizations to
superspace. We now describe the analogues of the elementary, homogeneous, and power sum symmetric functions.  Later in the section, we will introduce the analogues of the 
Schur symmetric functions.
\subsection{Elementary, homogeneous, and power sum symmetric functions}
\begin{itemize}
\item The power-sum 
symmetric
functions in superspace are  $p_\Lambda=\tilde{p}_{\Lambda_1}\cdots\tilde{p}_{\Lambda_m}p_{\Lambda_{m+1}}\cdots p_{\Lambda_{\ell}}\, $,
\begin{equation} 
\text{~~~{where}~~~}
\tilde{p}_k=\sum_{i=1}^N\theta_ix_i^k\qquad\text{and}\qquad p_r=
\sum_{i=1}^Nx_i^r \, , \text{~~~~for~~} k\geq 0,~r \geq 1;
\end{equation}  

\item 
The elementary
symmetric
functions in superspace are 
$e_\Lambda=\tilde{e}_{\Lambda_1}\cdots\tilde{e}_{\Lambda_m}e_{\Lambda_{m+1}}\cdots e_{\Lambda_{\ell}}$ ,

\begin{equation} \text{~~~where~~~} \tilde{e}_k=m_{(0;1^k)}
\quad \text{and} \quad e_r=m_{(\emptyset;1^r)}
, \text{~~~~for~~} k\geq 0,~r \geq 1;
\end{equation}  
\item
 The homogeneous
symmetric
functions in superspace are 
$h_\Lambda=\tilde{h}_{\Lambda_1}\cdots\tilde{h}_{\Lambda_m}h_{\Lambda_{m+1}}\cdots h_{\Lambda_{\ell}}\, ,$
\begin{equation}  \text{~~~where~~~}  \tilde{h}_k=\sum_{\Lambda \vdash (n|1)} (\Lambda_1+1)m_{\Lambda}
\qquad \text{and} \qquad h_r=\sum_{\Lambda \vdash (n|0)}m_{\Lambda}, 
 \text{~~~~for~~} k\geq 0,~r \geq 1
\end{equation}  
\end{itemize}
Observe that when $\Lambda=(\emptyset; \lambda)$, we have that $m_\Lambda=m_\lambda$,  $p_\Lambda=p_\lambda$,  $e_\Lambda=e_\lambda$ and
 $h_\Lambda=h_\lambda$ are respectively the usual monomial, power-sum, elementary and homogeneous symmetric functions.
Also note
that if we define the operator $d=\theta_1 {\partial}/{\partial_{x_1}}+\cdots + \theta_N \partial/\partial_{x_N}$, we have
\begin{equation} \label{opd}
(k+1) \, \tilde p_k = d ( p_{k+1})\, , \qquad  \tilde e_k = d (e_{k+1})  \quad \text{and} \quad \tilde h_k = d(h_{k+1})  
\end{equation}

\subsection{Scalar product}
The scalar product that we will consider generalizes naturally the usual Hall scalar product.  It is also best defined on power-sum symmetric functions.

Let $\langle \! \langle \, \cdot \, , \, \cdot \, \rangle \! \rangle : {\bf \Lambda} \times  {\bf \Lambda} \to \mathbb Z $
  be defined as
\begin{equation} \label{scalarprod}
\langle \! \langle \, p_\Lambda \, , \, p_{\Omega} \, \rangle \! \rangle = \delta_{\Lambda \Omega} \, z_{\Lambda^s}
\end{equation}
where \qquad  $z_{\Lambda^s} = 1^{n_{\Lambda^s}(1)} n_{\Lambda^s}(1)! 2^{n_{\Lambda^s}(2)} n_{\Lambda^s}(2)! \cdots$  \text{with} $n_{\Lambda^s}(i)$ the number of parts of ${\Lambda^s}$ equal to $i$.  The monomial and homogeneous symmetric functions are dual with respect to that scalar product, that is,
\begin{equation} \label{dualhm}
\langle \! \langle  h_\Lambda \, , \, m_\Omega  \rangle \! \rangle = \delta_{\Lambda \Omega}
\end{equation}

We also define the endomorphism $\omega$ as the unique homomorphism such
that
\begin{equation}
  \omega (p_r) = (-1)^{r-1} p_r\quad {\rm and} \qquad \omega (\tilde p_\ell) =
  (-1)^{\ell} \tilde p_\ell 
\end{equation}  
for $r=1,2,\dots$ and $\ell=0,1,2,\dots$.  The endomorphism $\omega$ is
then easily seen to be an involution as well as an isometry of the scalar product \eqref{scalarprod}, that is,
$\langle \! \langle \omega f \, , \, \omega g  \rangle \! \rangle=\langle \! \langle f \, , \, g \rangle \! \rangle$ for all $f,g \in {\bf \Lambda}$.
It is known moreover that
\begin{equation} \label{dualeh}
  \omega(e_\Lambda)= h_\Lambda
\end{equation}

\subsection{Schur functions in superspace}  There are two genuine families of Schur functions in superspace, denoted $s_\Lambda$ and $\bar s_\Lambda$, which can
be defined as generating sums of new types of tableaux. But because we will only need to use a few of their properties,
we will simply define them (even
though it is not very explicit) as special cases of Macdonald polynomials in superspace.

The Macdonald polynomials in superspace $\{P_{\Lambda}^{(q,t)}\}_{\Lambda}$ can be defined as the unique basis of the space
of symmetric functions in superspace such that
\begin{equation} \label{defjack}
  \begin{split}
& P_{\Lambda}^{(q,t)} = m_\Lambda + \text{smaller terms} \\
& \langle \langle P_{\Lambda}^{(q,t)}, P_{\Omega}^{(q,t)} \rangle \rangle_{q,t} = 0 \quad \text{ if } \Lambda \neq \Omega
  \end{split}
  \end{equation}
where the triangularity is with respect to the dominance ordering on superpartitions and where the scalar product $\langle \langle \cdot, \cdot \rangle \rangle_{q,t}$ is defined on power-sum symmetric functions as
\begin{equation}
  \langle \langle p_{\Lambda}, p_{\Omega} \rangle \rangle_{q,t} =
  \delta_{\Lambda \Omega} \, q^{|\Lambda^a|}  z_{\Lambda^s} \prod_{i=1}^{\ell(\Lambda^s)} \frac{1-q^{\Lambda^s_i}}{1-t^{\Lambda^s_i}}\, ,
  \qquad \qquad 
z_\lambda =\prod_{i\geq 1} i^{n_i(\lambda)} n_i(\lambda)!
\end{equation}  
with $m$ the fermionic degree of $\Lambda$ and $n_i(\lambda)$ the number of parts equal to $i$ in the partition $\lambda$.

Although the limiting cases $q=t=0$ and $q=t=\infty$ of the scalar product 
$\langle \langle \cdot, \cdot \rangle \rangle_{q,t}$ are degenerate and not well-defined respectively, the corresponding limiting cases
of the Macdonald polynomials in superspace exist and are combinatorially
very rich.  We thus
define the Schur functions in superspace $s_\Lambda$ and $\bar s_\Lambda$ as:
\begin{equation}\
s_\Lambda = P_\Lambda^{(0,0)} \qquad {\rm and} \qquad \bar s_\Lambda = P_\Lambda^{(\infty,\infty)} 
\end{equation}  

The following properties of the Schur functions in superspace can be found in
\cite{BDLM2,JL}.
\begin{proposition} \label{propdual}
Let $s_{\Lambda}^*$ and $\bar s_{\Lambda}^*$ be the bases dual to the bases  $s_{\Lambda}$ and $\bar s_{\Lambda}$ respectively, that is, let $s_{\Lambda}^*$ and $\bar s_{\Lambda}^*$ be such that
\begin{equation} 
\langle \! \langle s_{\Lambda}^*, s_{\Omega} \rangle \! \rangle =
\langle \! \langle \bar s_{\Lambda}^*, \bar s_{\Omega} \rangle \! \rangle =
 \delta_{\Lambda \Omega}
\end{equation}
Then
\begin{equation}
s_{\Lambda}^*= (-1)^{\binom{m}{2}} \omega \bar s_{\Lambda'} \qquad  {\rm and} \qquad
\bar s_{\Lambda}^*= (-1)^{\binom{m}{2}} \omega  s_{\Lambda'}
\end{equation}
where $m$ is the fermionic degree of $\Lambda$.
\end{proposition}

The skew-Schur functions in superspace can be defined as in the non-supersymmetric case.   Let $ s_{\Lambda/\Omega}$ and $\bar s_{\Lambda/\Omega}$ be defined such that
\begin{equation}
\langle \! \langle s_{\Omega}^* \, f, s_{\Lambda} \rangle \! \rangle =
\langle \! \langle f , s_{\Lambda/\Omega} \rangle \! \rangle   \qquad {\rm and} \qquad
 \langle \! \langle   \bar s_{\Omega}^* \, f, \bar s_{\Lambda} \rangle \! \rangle =
\langle \! \langle f , \bar s_{\Lambda/\Omega} \rangle \! \rangle
\end{equation}
for all symmetric functions in superspace $f$.

We also define the analogs of  the Littlewood-Richardson coefficients
$\bar c^\Lambda_{\Gamma \Omega}$ and $c^\Lambda_{\Gamma \Omega}$ to be respectively such that 
\begin{equation}  \label{LR}
\bar s_{\Gamma}\,  \bar s_{\Omega} = \sum_{\Lambda} \bar c^\Lambda_{\Gamma \Omega}\,  \bar s_{\Lambda}
 \qquad {\rm and} \qquad    
s_{\Gamma}\,  s_{\Omega} = \sum_{\Lambda} c^\Lambda_{\Gamma \Omega}\,  s_{\Lambda}    
\end{equation}
Note that it is immediate from the (anti-)commutation relations between the Schur functions in superspace
 that if $\Gamma$ and $\Omega$ are respectively of fermionic degrees $a$ and $b$, then 
$\bar c^\Lambda_{\Gamma \Omega}=(-1)^{ab} \,  \bar c^\Lambda_{\Omega \Gamma}$ and 
 $c^\Lambda_{\Gamma \Omega}=(-1)^{ab}  \, c^\Lambda_{\Omega \Gamma}$.

The well-known connection between 
Littlewood-Richardson coefficients and skew Schur functions extends to
superspace.
\begin{proposition}  \label{propskew} We have
\begin{equation}
s_{\Lambda/\Omega} =  \sum_{\Gamma}  \bar c^{\Lambda'}_{\Gamma' \Omega'} \, s_\Gamma
\qquad {\rm and} \qquad  
\bar s_{\Lambda/\Omega} = \sum_{\Gamma}  c^{\Lambda}_{\Omega \Gamma} \, \bar s_\Gamma
\end{equation}
Furthermore, 
$c^\Lambda_{\Omega \Gamma}=c^{\Lambda'}_{\Gamma' \Omega'}$ (while we note that
$\bar c^\Lambda_{\Omega \Gamma}\neq \bar c^{\Lambda'}_{\Gamma' \Omega'}$ in general).
\end{proposition}

\section{The Hopf algebra of symmetric functions in superspace}

In this section we show that the ring of symmetric functions in superspace
${\bf \Lambda}$ has a Hopf algebra structure which extends naturally that
of the usual symmetric functions (see for instance \cite{LMW} and
\cite{GR}).

\subsection{Hopf algebra structure of ${\bf \Lambda}$}
As mentioned before, the ring ${\bf \Lambda}$ has a natural grading,
called the fermionic degree,
which counts the degree in the anticommuting variables of the functions.
It is easy to check that
\begin{equation}
f g = (-1)^{a b} gf
\end{equation}  
if $f$ and $g$ have fermionic degrees $a$ and $b$ respectively.

Extending what is usually done in the symmetric function case \cite{M}, we will identify
${\bf \Lambda} \otimes_{\mathbb Q} {\bf \Lambda}$ (which from now on we will denote ${\bf \Lambda} \otimes {\bf \Lambda}$ for simplicity) with symmetric functions of two sets of variables
$(x_1,x_2,\dots;\theta_1,\theta_2,\dots)$ and $(y_1,y_2,\dots;\phi_1,\phi_2,\dots)$,
with
the extra requirement that the variables $\theta$ and $\phi$ anticommute,
that is, $\theta_i \phi_j=-\phi_j \theta_i$.  This way,
$f \otimes g$ corresponds to $f(x;\theta) g(y;\phi)$ and
${\bf \Lambda} \otimes {\bf \Lambda}$ becomes an algebra with a product satisfying the relation
\begin{equation}
  (f_1\otimes g_1) \cdot  (f_2\otimes g_2)=
   (-1)^{ab}  f_1 f_2 \otimes g_1 g_2
    \end{equation}
  for $f_1,f_2,g_1,g_2 \in {\bf \Lambda}$ with $g_1$ and $f_2$ of fermionic degree $a$ and $b$ respectively.

 The comultiplication
$\Delta: {\bf \Lambda} \to {\bf \Lambda} \otimes {\bf \Lambda} $ is defined as
\begin{equation}
(\Delta f)(x,y;\theta,\phi) = f(x,y;\theta,\phi)
\end{equation}  
where as we just mentioned, $f(x,y;\theta,\phi)$ is considered to be an element
of ${\bf \Lambda} \otimes {\bf \Lambda} $. 
With this definition, the coproduct is immediately coassociative
\begin{equation}
  (\Delta \otimes {\rm id}) \circ \Delta f= f(x,y,z;\theta,\phi,\varphi)=
   ({\rm id} \otimes \Delta) \circ \Delta f
\end{equation}
and an algebra morphism
\begin{equation} \label{dalgmor}
  \Delta(fg)= \Delta(f) \cdot \Delta(g) \qquad {\rm and} \qquad
  \Delta\bigl(u(1)\bigr)=\Delta(1)=1 \otimes 1
\end{equation}
since $(fg)(x,y;\theta,\phi)=f(x,y;\theta,\phi) g(x,y;\theta,\phi)$
for all $f,g \in {\bf \Lambda}$.

It is straightforward to check that the $p_i$'s and $\tilde p_i$'s are primitive elements, that is,  
\begin{equation} \label{eqcoprodp}
  \Delta p_i = p_i \otimes 1 + 1 \otimes p_i \quad {\rm and} \quad
   \Delta \tilde p_i = \tilde p_i \otimes 1 + 1 \otimes \tilde p_i 
\end{equation}
The coproduct of the elementary and homogeneous symmetric functions also has a simple expression.
\begin{proposition} We have
\begin{equation} \label{coprode}
  \Delta e_i = \sum_{k+\ell=i} e_k \otimes e_\ell 
  \quad {\rm and} \quad
  \Delta \tilde e_i = \sum_{k+\ell=i} 
 \left( \tilde e_k \otimes e_\ell + e_\ell \otimes \tilde e_k \right)
\end{equation}  
and
\begin{equation} \label{eqcoprodh}
  \Delta h_i = \sum_{k+\ell=i} h_k \otimes h_\ell 
  \quad {\rm and} \quad
  \Delta \tilde h_i = \sum_{k+\ell=i} 
 \left( \tilde h_k \otimes h_\ell + h_\ell \otimes \tilde h_k \right)
\end{equation}  
\end{proposition}
\begin{proof}  The formulas for $ \Delta e_i$ and $ \Delta h_i$ are well
  known \cite{M}.   We will use the operator $d$ that appears in \eqref{opd}
  to derive the formulas for $ \Delta \tilde e_i$ and $ \Delta \tilde h_i$.
  On symmetric functions in superspace, the operator $d$ can be defined as the
  unique linear operator such that
  \begin{equation}
   d(p_{k+1})=(k+1)\, \tilde  p_k\, , \qquad d(\tilde p_k)=0   \qquad k=0,1,2,\dots
  \end{equation}
  and such that
  \begin{equation}
   d(fg) = d(f) g +(-1)^a fd(g)
  \end{equation}  
  whenever $f$ is of fermionic degree $a$.  We will now see that
  \begin{equation} \label{reld}
  \Delta \circ d = (d \otimes 1 + 1 \otimes d) \circ \Delta
  \end{equation}  
  The relation can be checked to hold when it acts on $p_i$ or $\tilde p_i$ by \eqref{eqcoprodp}. Hence
  \begin{equation} \label{reldd}
 \Delta \circ d (p_\Lambda) = (d \otimes 1 + 1 \otimes d) \circ \Delta (p_\Lambda)
\end{equation}
  for any $\Lambda$ of length 1.
  Supposing by induction that \eqref{reldd} holds for $p_\Lambda$ with
  $\Lambda$ of
  length $n-1$, we have
  \begin{align}
    \Delta \circ d(p_\Lambda p_k) & = \Delta \bigl( d(p_\Lambda) p_k +(-1)^m p_\Lambda d(p_k) \bigr) \nonumber \\
    & = \Delta \bigl(d(p_\Lambda) \bigr) \cdot \Delta( p_k)  +(-1)^m \Delta(p_\Lambda) \cdot \Delta\bigl( d(p_k) \bigr)  \nonumber \\
    & = \bigl( (d \otimes 1 + 1 \otimes d) \circ \Delta(p_\Lambda) \bigr)\cdot \Delta( p_k)  +(-1)^m \Delta(p_\Lambda)\cdot  (d \otimes 1 + 1 \otimes d) \circ \Delta (p_k)  \nonumber \\
    & =   (d \otimes 1 + 1 \otimes d) \circ \Delta (p_\Lambda p_k)
  \end{align}
  where we assumed without loss of generality that $\Lambda$ is of fermionic
  degree $m$.
  Similarly, since $d(\tilde p_k)=0$, we have again by induction that
   \begin{align}
    \Delta \circ d(p_\Lambda \tilde p_k) & = \Delta \bigl( d(p_\Lambda) \tilde p_k  \bigr) \nonumber \\
    & = \Delta \bigl(d(p_\Lambda) \bigr) \cdot \Delta( \tilde p_k)     \nonumber \\
    & = \bigl( (d \otimes 1 + 1 \otimes d) \circ \Delta(p_\Lambda) \bigr)\cdot \Delta( \tilde p_k)    \nonumber \\
    & =   (d \otimes 1 + 1 \otimes d) \circ \Delta (p_\Lambda \tilde p_k)
    \end{align}
Hence \eqref{reldd} holds for any superpartition $\Lambda$ of length $n$, which proves \eqref{reld} by induction.

Using \eqref{opd}, we thus have
\begin{align}
  \Delta \tilde e_i = \Delta \circ d(e_{i+1})=  (d \otimes 1 + 1 \otimes d) \circ \Delta (e_{i+1}) & =  (d \otimes 1 + 1 \otimes d) \sum_{k+\ell=i+1} e_k \otimes e_\ell \nonumber \\
  & = \sum_{k+\ell=i+1} (\tilde e_{k-1} \otimes e_\ell + e_{k} \otimes \tilde e_{\ell-1})  
  \end{align}  
which proves the formula for $\Delta \tilde e_i$.  The formula for
$\Delta \tilde h_i$ can be deduced from that of $\Delta h_{i+1}$ in exactly the same way.
\end{proof}

We now prove that ${\bf \Lambda}$ is a Hopf algebra.  Define
the counit $\epsilon: {\bf \Lambda} \to \mathbb Q$ to be
the identity on ${\bf \Lambda}^{0}$ and  the null operator on
${\bf \Lambda}^{n}$ for $n > 0$, where we use the grading defined
in \eqref{newgrad}.  Since the counit is easily seen to be an algebra morphism,
we have that  ${\bf \Lambda}$ is a bialgebra by \eqref{dalgmor}.
Furthermore, ${\bf \Lambda}$  is a graded bialgebra since
$\Lambda^\ell \Lambda^n \subseteq \Lambda^{\ell+n}$ and
$\Delta(\Lambda^n) \subseteq \bigoplus_{k+\ell=n} \Lambda^k \otimes \Lambda^\ell$,
the latter property being a consequence of \eqref{eqcoprodp}.  
Moreover, ${\bf \Lambda}$ is connected given that ${\bf \Lambda}^0=\mathbb Q$.
Therefore, from Theorem~\ref{hopftheo}, ${\bf \Lambda}$ is automatically
a Hopf algebra (the antipode will
be described explicitly later).

In order to obtain the coproduct of the Schur functions in superspace,
we now prove a proposition expressing how
$s_\Lambda(x,y;\theta,\phi)$ splits into Schur functions in superspace of each alphabet.
It relies on the use of the following Cauchy-type identities
\begin{equation} \label{cauchy}
\prod_{i,j} \frac{1}{(1-x_i y_j -\theta_i \phi_j)}= \sum_{\Lambda} s_\Lambda(x;\theta) s^*_\Lambda(y;\phi)= \sum_{\Lambda} \bar s_\Lambda(x;\theta) \bar s^*_\Lambda(y;\phi)
\end{equation}  
which are consequences of the duality in Proposition~\ref{propdual} (see \cite{BDLM2}).
\begin{proposition} We have
  \begin{equation}
     s_\Lambda(x,y;\theta,\phi)=  \sum_\Omega s_{\Lambda/\Omega}(x;\theta) s_{\Omega}(y;\phi) \qquad {\rm and} \qquad  \bar s_\Lambda(x,y;\theta,\phi)=  \sum_\Omega \bar s_{\Lambda/\Omega}(x;\theta) \bar s_{\Omega}(y;\phi) 
   \end{equation} 
 \end{proposition}  
\begin{proof}  We will prove only the first formula since the other one can be proved in exactly the same way.
  Using \eqref{cauchy}, we have, on the one hand
  \begin{equation} \label{eqh1}
    \prod_{i,j} \frac{1}{(1-x_i z_j -\theta_i \varphi_j)}
    \prod_{k,\ell} \frac{1}{(1-y_k z_\ell -\phi_k \varphi_\ell)}= \sum_{\Lambda} s_\Lambda(x,y;\theta,\phi) s^*_\Lambda(z;\varphi)
      \end{equation}  
  and, on the other hand,
 \begin{equation} \label{eqh2}
    \prod_{i,j} \frac{1}{(1-x_i z_j -\theta_i \varphi_j)}
    \prod_{k,\ell} \frac{1}{(1-y_k z_\ell -\phi_k \varphi_\ell)}= \sum_{\Omega,\Gamma} s_\Omega(x;\theta) s^*_\Omega(z;\varphi)  s_\Gamma(y;\phi) s^*_\Gamma(z;\varphi) 
      \end{equation}  
 Now, applying the endomorphism $\omega$ on the first equation of \eqref{LR} gives
 \begin{equation} \label{eqprodst}
   (-1)^{\binom{a}{2}+\binom{b}{2}}  s^*_{\Gamma'} s^*_{\Omega'} = \sum_{\Lambda}  (-1)^{\binom{c}{2}} \bar c_{\Gamma \Omega}^{\Lambda} s^*_{\Lambda'} \iff
 (-1)^{ab}  s^*_{\Gamma} s^*_{\Omega} = \sum_{\Lambda}  \bar c_{\Gamma' \Omega'}^{\Lambda'} s^*_{\Lambda}
    \end{equation}  
 where $a,b$ and $c$ are the fermionic degrees of $\Gamma, \Omega$ and $\Lambda$ respectively.  In the equivalence, we used the fact that
 \begin{equation}
\binom{a}{2}+\binom{b}{2}+ab=\binom{a+b}{2}=\binom{c}{2}
    \end{equation}  
since $c=a+b$ (otherwise $ \bar c_{\Omega' \Gamma'}^{\Lambda'}=0$).
 
  From \eqref{eqh1} and \eqref{eqh2}, we thus get
 \begin{equation} \label{eqh3}
   \sum_{\Lambda} s_\Lambda(x,y;\theta,\phi) s^*_\Lambda(z;\varphi)=
   \sum_{\Omega,\Gamma,\Lambda} \bar c_{\Omega' \Gamma'}^{\Lambda'}
   s_\Omega(x;\theta) s_\Gamma(y;\phi) s^*_\Lambda(z;\varphi) 
    \end{equation}  
 where we used the relation
 \begin{equation}
 s^*_\Omega(z;\varphi)  s_\Gamma(y;\phi)=(-1)^{ab} s_\Gamma(y;\phi) s^*_\Omega(z;\varphi) 
 \end{equation}  

Using Proposition~\ref{propskew} in \eqref{eqh3} then gives
\begin{equation} 
   \sum_{\Lambda} s_\Lambda(x,y;\theta,\phi) s^*_\Lambda(z;\varphi)=
   \sum_{\Gamma,\Lambda} 
   s_{\Lambda/\Gamma}(x;\theta) s_\Gamma(y;\phi) s^*_\Lambda(z;\varphi) 
    \end{equation} 
which proves the proposition.
\end{proof}
The coproducts of $s_\Lambda$ and $\bar s_\Lambda$ can now be given explicitly.
\begin{corollary} \label{corocoschur}
  We have
  \begin{equation}
    \Delta s_\Lambda=  \sum_\Omega s_{\Lambda/\Omega} \otimes s_{\Omega}
    \qquad {\rm and} \qquad  \Delta \bar s_\Lambda=  \sum_\Omega \bar s_{\Lambda/\Omega} \otimes \bar s_{\Omega}
  \end{equation}
Or, equivalently by Proposition~\ref{propskew},
\begin{equation}
\Delta s_\Lambda= \sum_{\Omega,\Gamma}\bar c^{\Lambda'}_{\Gamma' \Omega'} \, s_\Gamma  \otimes s_{\Omega}  \qquad {\rm and} \qquad \Delta \bar s_\Lambda= \sum_{\Omega,\Gamma} c^{\Lambda}_{ \Omega \Gamma} \, \bar s_\Gamma  \otimes \bar s_{\Omega}  
\end{equation}
\end{corollary}  
Given that $c^{\Lambda'}_{\Gamma' \Omega'}= (-1)^{ab}c^{\Lambda'}_{\Omega' \Gamma'}$
if $a$ and $b$ are the fermionic degrees of ${\Gamma}$ and ${\Omega}$
respectively, the previous corollary immediately implies that
\begin{equation} \label{eqtwist}
\tau \circ \Delta f = \Delta f 
\end{equation}  
where the twist map $\tau: \mathcal H \otimes \mathcal H \to \mathcal H \otimes \mathcal H $ is such that $\tau (g \otimes h)=(-1)^{ab} (h \otimes g)$ if $g$ and $h$ are respectively of fermionic degrees $a$ and $b$ (the topologist's twist map introduced in Section~\ref{SecHopf}).
Hence $\Lambda$ is a {\defstyle cocommutative} Hopf
algebra (in the topologist's sense due to the extra signs).

\subsection{The antipode and the involution $\omega$} The Hopf algebra ${\bf \Lambda}$ has a unique antipode $S: {\bf \Lambda} \to {\bf \Lambda}$ which is such that $S(a)=-a$ if $a$ is a primitive element by  \eqref{eqantipo}.  Since  power-sums,
$p_i,\tilde p_j$ for $i \geq 1$ and $j \geq 0$ are primitive generators of ${\bf \Lambda}$, $S$ can be defined by its action on the powers-sums:
\begin{equation} \label{eqantipode}
S(p_i)=-p_i, \quad S(\tilde p_j)=- \tilde p_j \implies
  S(p_\Lambda) = (-1)^{\ell(\Lambda)} p_\Lambda
\end{equation}
since $S(\tilde p_{\Lambda_1}\cdots \tilde p_{\Lambda_m})=(-1)^{\binom{m}{2}} S(\tilde p_{\Lambda_m})\cdots S(\tilde p_{\Lambda_1})=(-1)^{m+\binom{m}{2}} \tilde p_{\Lambda_m}\cdots \tilde p_{\Lambda_1}=(-1)^{m} \tilde p_{\Lambda_1}\cdots \tilde p_{\Lambda_m}$.  Note that the sign in the first equality 
stems from the fact that $S$ is a signed anti-homomorphism,
that is, that it satisfies the relation  $S(fg)=(-1)^{ab}S(g)S(f)$ for $f,g \in {\bf \Lambda}$ of fermionic degrees $a$ and $b$ respectively.

The antipode $S$ connects with the involution $\omega$ in the following way.
\begin{proposition} \label{propanti} We have that
  \begin{equation}
     S(f) = (-1)^{m+n} \omega(f)
  \end{equation}  
if $f \in {\bf \Lambda}_{n,m}$.
\end{proposition}  
\begin{proof}
From the definition of $\omega$, we have
\begin{equation}  
\omega(p_\Lambda)= (-1)^{|\Lambda|-\ell(\Lambda^s)} p_\Lambda
\end{equation}
The result thus holds since
$\ell(\Lambda)=\ell(\Lambda^s)+m$ and $|\Lambda|=n$
imply that
\begin{equation}
(-1)^{m+n} \omega(p_\Lambda)= (-1)^{\ell(\Lambda)} p_\Lambda= S(p_\Lambda)
\end{equation}  
  as we just saw in \eqref{eqantipode}.
\end{proof}
 Proposition~\ref{propdual} and \eqref{dualeh} then immediately give
 \begin{corollary} If $\Lambda$ is a superpartition of total degree $n$ and fermionic degree $m$, then the antipode $S$ is such that
  \begin{equation}
    S(e_\Lambda)= (-1)^{m+n} h_\Lambda \,, \qquad S(s_\Lambda)=(-1)^{\binom{m+1}{2}+n} \bar s_{\Lambda'}^*  \,, \qquad {\rm and}  \qquad \qquad S(\bar s_\Lambda)
    =(-1)^{\binom{m+1}{2}+n} s^*_{\Lambda'} 
  \end{equation}  
  where $s_\Lambda^*$ and $\bar s_\Lambda^*$ are the dual bases to $s_\Lambda$ and
  $\bar s_\Lambda$ respectively (see Proposition~\ref{propdual}).
\end{corollary}  

\subsection{Self-duality}
The scalar product on ${\bf \Lambda}$ defined in \eqref{scalarprod}
can be extended to ${\bf \Lambda} \otimes {\bf \Lambda}$.
\begin{definition} The ring ${\bf \Lambda} \otimes {\bf \Lambda}$ has a scalar product defined as
  \begin{equation}
    \langle \! \langle f_1\otimes g_1, f_2\otimes g_2 \rangle \! \rangle
    = (-1)^{ab} \langle  \! \langle f_1,f_2 \rangle  \! \rangle    \langle \! \langle g_1,g_2\rangle  \! \rangle
  \end{equation}
  for $f_1,f_2,g_1,g_2 \in {\bf \Lambda}$ with $g_1$ and $f_2$ of fermionic degree
  $a$ and $b$ respectively.
\end{definition}
We should note that this is simply the pairing described in \eqref{scaltau}
in the case where $\mathcal H= \mathcal H'={\bf \Lambda}$.

The following proposition implies that the Hopf algebra ${\bf \Lambda}$ is
self-dual (in the topologist's sense) given that the other conditions in \eqref{selfdual} are trivially satisfied (the one involving the antipode follows from \eqref{eqantipode}).
\begin{proposition} We have
  \begin{equation}
  \langle \! \langle \Delta f, g \otimes h \rangle \! \rangle
  =  \langle \! \langle f, g  h \rangle \! \rangle
  \end{equation}
  for $f,g,h \in {\bf \Lambda}$.
\end{proposition}  
\begin{proof}
  It suffices to show that
  \begin{equation}
     \langle \! \langle \Delta s_\Lambda, s_\Omega^* \otimes s_\Gamma^* \rangle \! \rangle
  =  \langle \! \langle s_\Lambda, s_\Omega^*  s_\Gamma^* \rangle \! \rangle
   \end{equation}  
  From Corollary~\ref{corocoschur} and Proposition~\ref{propskew}, we have
  that
\begin{equation}
  \langle \! \langle \Delta s_\Lambda, s_\Omega^* \otimes s_\Gamma^* \rangle \! \rangle =  \langle \! \langle  \sum_\Delta s_{\Lambda/\Delta} \otimes s_{\Delta} , s_\Omega^* \otimes s_\Gamma^* \rangle \! \rangle = (-1)^{ab} \langle \! \langle s_{\Lambda/\Gamma} , s_\Omega^*  \rangle \! \rangle = (-1)^{ab}\bar c_{\Omega' \Gamma'}^{\Lambda'}
  \end{equation}
where $a$ and $b$ are the fermionic degrees of $\Omega$ and $\Gamma$ respectively.  On the other hand, if we use \eqref{eqprodst}, we get
\begin{equation}
 \langle \! \langle s_\Lambda, s_\Omega^*  s_\Gamma^* \rangle \! \rangle=  (-1)^{ab}\bar c_{\Omega' \Gamma'}^{\Lambda'}
\end{equation}  
and the proposition follows.
  \end{proof}  
We have thus proven the following proposition.
\begin{proposition} The ring ${\bf \Lambda}$ of symmetric functions in superspace is a cocommutative and self-dual Hopf algebra (in the topologist's sense).
\end{proposition}

\section{The Hopf algebra of quasisymmetric functions in superspace}

Before introducing the ring of quasisymmetric functions in superspace,
we define our analogues of compositions.
  \begin{definition}
  A {\defstyle dotted composition} $(\alpha_1,\alpha_2,\dots,\alpha_l)$
  is a vector whose entries either belong to  $\{1,2,3,\dots\}$ or 
  to $\{\dot 0, \dot 1,\dot 2,\dots\}$.
The {\defstyle length} of $\ga$,
denoted $\ell(\ga)$, is the number of parts $l$ of $\alpha$.
We define the
sequence $\eta=\eta(\ga)=(\eta_1,\ldots,\eta_{\ell(\ga)})$ by
\begin{equation}\label{D:eta}
  \eta_i=\begin{cases}1&\text{if $\ga_i$ is dotted,}\\0&\text{otherwise.}\end{cases}
\end{equation}
We let $|\alpha|:=\alpha_1+\cdots +\alpha_l$
be the {\defstyle total degree} of $\alpha$ (in the sum, the dotted entries
are considered as if they did not have dots on them).
The number of dotted parts of $\alpha$ is called the {\defstyle fermionic degree} of $\ga$. We write $x^{\ga_i}_j$ whether
 $\ga_i$ is dotted or not.
\end{definition}

  The definition of the ring of quasisymmetric functions in superspace
  then extends naturally that of the usual quasisymmetric functions \cite{Gessel,GR,LMW}.

  \begin{definition}
  \label{D:qs}
Let $\mathcal R(x,\theta)$ be the ring of formal power series of finite degree in
$\Q[[x_1,x_2,\ldots,\te_1,\te_2,\ldots]]$.
The {\defstyle quasisymmetric functions in superspace} ${\rm sQSym}$
will be the $\mathbb Q$-vector space of the elements
$f$ of $\mathcal R(x,\theta)$ such that
for every dotted compositions
  $\ga=(\ga_1,\ldots,\ga_{\ell})$ with $\eta=\eta(\ga)$ as in
  \eqref{D:eta}, all monomials
  $\te_{i_1}^{\eta_1}\cdots\te_{i_{\ell}}^{\eta_{\ell}}x^{\ga_1}_{i_1}\cdots
  x^{\ga_{\ell}}_{i_{\ell}}$ in $f$ with indices $i_1<\cdots
  <i_{\ell}$ have the same coefficient.
\end{definition}

It is easy to see that ${\rm sQSym}$ is bigraded with respect of the total degree and the fermionic degree, that is,
\begin{equation}
{\rm sQSym} = \bigoplus_{n,m} {\rm sQSym}_{n,m}
\end{equation}  
where ${\rm sQSym}_{m,n}$ is the subspace of quasisymmetric functions in superspace of total degree $n$ and fermionic degree $m$.

\subsection{Monomial quasisymmetric functions in superspace}

There is a natural basis of  ${\rm sQSym}$ provided by the
generalization of the monomial quasisymmetric functions to superspace.
\begin{definition}
  \label{D:msq}
 Let $\alpha$ be a dotted composition with $\ell(\ga)=l$.  Then the {\defstyle monomial quasisymmetric
   function in superspace} $M_\ga$ is defined
 as
 \begin{equation}
   M_{\alpha}=\sum_{i_1<i_2<\cdots<i_l}
 \te_{i_1}^{\eta_1} \te_{i_2}^{\eta_2} \cdots \te_{i_l}^{\eta_l} x_{i_1}^{\ga_1}\cdots
 x_{i_l}^{\ga_l}
 \end{equation}
 where $\eta=\eta(\ga)$.
\end{definition}

\begin{example}
 Restricting to four variables, we have
 $$M_{\dot{3},1,2}(x_1,x_2,x_3,x_4;\te_1,\te_2,\te_3,\te_4)=\te_1x_1^3x_2x_3^2+\te_1x_1^3x_2x_4^2+\te_1x_1^3x_3x_4^2+\te_2x_2^3x_3x_4^2,$$
and
$$M_{3,\dot{1},\dot{2}}(x_1,x_2,x_3,x_4;\te_1,\te_2,\te_3,\te_4)=\te_2\te_3 x_1^3x_2x_3^2+\te_2\te_4 x_1^3x_2x_4^2+\te_3\te_4 x_1^3x_3x_4^2+\te_3\te_4 x_2^3x_3x_4^2.$$
\end{example}
For a dotted composition $\alpha$,
we will say that the term $\te_{1}^{\eta_1} \te_{2}^{\eta_2} \cdots \te_{l}^{\eta_l} x_{1}^{\ga_1}\cdots x_{l}^{\ga_l}$ is the {\defstyle leading term} of $M_\alpha$.
By symmetry, it is obvious that it suffices to know the coefficients of the leading terms that appear in a given $f  \in {\rm sQSym}$ in order to get its full expansion
in monomial quasisymmetric functions in superspace.

The ring of symmetric functions belongs to  ${\rm sQSym}$ since
the monomial symmetric function $m_{\Lambda}$ expands in the following
way in terms of $M_\alpha$'s:
\begin{equation} \label{menM}
m_\Lambda = \sum_{\alpha \, : \, \tilde \alpha=\gamma} (-1)^{\sigma(\alpha)} M_\alpha
\end{equation}  
where $\gamma$ is the dotted composition $(\dot \Lambda^a_1, \dots, \dot \Lambda^a_m,\Lambda^s_1,\dots,\Lambda^s_l)$ obtained from $\Lambda$, and where
$\tilde \alpha= \gamma$ whenever the entries of $\alpha$
rearrange to $\gamma$. Finally, $\sigma(\alpha)$ is the sign of the permutation needed to reorder the dotted entries of $\alpha$ (read from left to right) to
$(\dot \Lambda^a_1, \dots, \dot \Lambda^a_m)$.

We will now see that ${\rm sQSym}$ is a also ring.  For this purpose, we first need to understand how monomials in superspace multiply.
Let $\ga=(\ga_1,\ldots , \ga_l)$ and $\gb=(\gb_1,\ldots ,\gb_k)$ be
two dotted compositions.
The product rule is similar to the non-super case, with only the
addition of a sign. We begin our explanation with the
consideration of a typical product of two monomials $Q_1$ and $Q_2$ in
$M_{\ga}M_{\gb}$ giving rise to a leading term:

$$Q_1 Q_2=(\te^{\eta_1}_{i_1}\cdots\te^{\eta_{\ell}}_{i_{\ell}}x^{\ga_1}_{i_1}\cdots x^{\ga_{\ell}}_{i_{\ell}})(\te^{\eta'_1}_{i'_1}\cdots\te^{\eta'_{\ell'}}_{i'_{\ell'}}x^{\gb_1}_{i'_1}\cdots x^{\gb_{\ell'}}_{i'_{\ell'}})
= (-1)^s \te^{\mu_1}_{1}\cdots\te^{\mu_{r}}_{{r}}x^{\gamma_1}_{1}\cdots x^{\gamma_{r}}_{{r}}
$$
for some sign $s$, where
$\eta=\eta(\ga)$, $\eta'=\eta(\gb)$ and $\mu=\eta(\gamma)$.
\omitt{
$$\eta_j=\begin{cases}1&\textrm{if $\ga_j$ is dotted}\\0&\textrm{if $\ga_j$ is not dotted}\\\end{cases}\text{ and }\eta'_j=\begin{cases}1&\textrm{if $\gb_j$ is dotted}\\0&\textrm{if $\gb_j$ is not dotted}\\\end{cases}$$}
Since $\te_i\te_j=-\te_j\te_i$, we must consider how the indices of
the monomials combine to determine the sign $s$.
If we let
$$S=S(\ga,\gb)=\{(p,q)|\ga_q\textrm{ is dotted in
  $\ga$, $\gb_p$ is dotted in $\gb$, and $i'_p<i_q$}\},$$
then it is easy to deduce that
$s=|S|$ since $s$ is the number of pairs $\te_{i_q}\te_{i'_p}$
which have to switch to $\te_{i'_p}\te_{i_q}$ when we put the
variables in increasing order. If $h=i_q=i'_p$ and both
$\ga_q$ and $\gb_p$ are dotted, then $Q_1 Q_2=0$.

The dotted composition $\gamma$ is given by
\begin{equation}\label{D:gamma}\gamma_h=\begin{cases}\ga_q&\text{if $h=i_q,h\neq i'_p$ all $p$},\\
\gb_p&\text{if $h=i'_p,h\neq i_q$ all $q$},\\
\ga_q+\gb_p&\text{if $h=i_q$ and $h=i'_p$},\\
\end{cases}\end{equation}
where in all cases, $\gamma_h$ is dotted if either $\ga_q$ or $\gb_p$ is.

\omitt{$$\mu_j=\begin{cases}1&\textrm{if $\gamma_j$ is dotted, and}\\0&\textrm{if $\gamma_j$ is not dotted.}\\\end{cases}$$}

As in \cite{LMW}, we can encode the pair $Q_1, Q_2$ as a path. We make a $\ell'$
by $\ell$ grid and label the rows by $\gb$ and the columns by $\ga$.
If both $\ga_q$ and $\gb_p$ are dotted, then place a dot in the cell
in row $p$ and column $q$. The path $P$ in the $(x,y)$ plane
from $(0,0)$ to $(\ell,\ell')$ with steps $(0,1)$, $(1,0)$, and
$(1,1)$ is similar to the paths defined in \cite[Section 3.3.1]{LMW}.
In our case, paths are not allowed to step diagonally over cells where
both $\ga_q$ and $\gb_p$ are dotted. The $h^{\text{th}}$ step of $P$
will be horizontal if $\gamma_h=\ga_q$, which is case one of
\eqref{D:gamma}, and it will be vertical if $\gamma_h=\gb_p$, which is
case two of \eqref{D:gamma}. Finally it will be diagonal in the third
case, where $\gamma_h=\ga_q+\gb_p$.  The path $P$ is in bijection with $Q_1$ and $Q_2$, and as such, the set of all paths determine all possible leading terms, or equivalently, all possible quasi-monomials that appear in the product.
We denote the dotted composition corresponding to the path $P$ by $\gamma=\Gamma(P)$. We call the set of all
paths which can be obtained from $\ga$ and $\gb$ in this manner {\defstyle
  the set of $(\ga,\gb)$ overlapping shuffles}.


Suppose $\ga_q$ and $\gb_p$ are both dotted and the path $P$ lies
above the $(p,q)$ cell. Then $P$ took the vertical step over row $p$
before taking the horizonal step over column $q$, meaning
$i'_p<i_q$. The pair $(p,q)$ is an element of $S$. Similarly, if $P$
lies below, then $(p,q)$ will not be an element of $S$. We have now
verified the following proposition.

\begin{proposition} \label{propmultM}
Suppose $\ga$ and $\gb$ are dotted compositions. Then
\begin{equation}
  M_{\ga}M_{\gb} =  \sum_P \sign(P) M_{\Gamma(P)}
  \end{equation}
where the sum is over all $(\ga,\gb)$ overlapping shuffles
and where
the {\defstyle
  sign} of the path  $P$ is given by
\begin{equation}
\sign(P)= (-1)^{\textrm{number of dots below the path }P}
\end{equation}  
\end{proposition}

\begin{example}
Let $\ga = (\dot{3},2)$ and $\gb = (\dot{4},1)$.

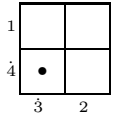
\begin{center}
\setlength{\unitlength}{0.6mm}
\begin{picture}(20,20)(0,0)
\put(0,0){\line(1,0){20}}
\put(0,10){\line(1,0){20}}
\put(0,20){\line(1,0){20}}
\put(0,0){\line(0,1){20}}
\put(10,0){\line(0,1){20}}
\put(20,0){\line(0,1){20}}

\put(5,5){\circle*{1.6}}

\put(3,-4){\tiny $\dot{3}$}
\put(13,-4){\tiny $2$}

\put(-3,4){\tiny $\dot{4}$}
\put(-3,14){\tiny $1$}

\end{picture}
\end{center}

Then $M_{\dot{3},2} M_{\dot{4},1} = M_{\dot{3},2,\dot{4},1} +
M_{\dot{3},\dot{6} ,1} + M_{\dot{3},\dot{4},2,1} +
M_{\dot{3},\dot{4},3} + M_{\dot{3},\dot{4},1,2} -
M_{\dot{4},\dot{3},2,1} - M_{\dot{4},\dot{3},3} -
M_{\dot{4},\dot{3},1,2} - M_{\dot{4},\dot{4},2} -
M_{\dot{4},1,\dot{3},2}$

\end{example}

\subsection{Hopf algebra structure of {\rm sQSym}}
Now that we have established that ${\rm sQSym}$ is an algebra, we will show that it is a also a Hopf algebra.
As we did earlier in the case of symmetric functions in superspace, we will identify
${\rm sQSym} \otimes_{\mathbb Q} {\rm sQSym}$ (which from now on will be denoted ${\rm sQSym} \otimes \, {\rm sQSym}$ for simplicity) with quasisymmetric functions in two sets of variables
$(x_1,x_2,\dots;\theta_1,\theta_2,\dots)$ and $(y_1,y_2,\dots;\phi_1,\phi_2,\dots)$,  where \quad $x_1<x_2 <
\ldots < y_1 < y_2 < \ldots$ and $\theta_1<\theta_2 <
\ldots < \phi_1 < \phi_2 < \ldots$, 
with
the extra requirement that the variables $\theta$ and $\phi$ anticommute.
This way,
$f \otimes g$ corresponds to $f(x;\theta) g(y;\phi)$ and
${\rm sQSym}  \otimes {\rm sQSym} $ becomes an algebra with a product satisfying the relation
\begin{equation}
  (f_1\otimes g_1) \cdot  (f_2\otimes g_2)=
   (-1)^{ab}  f_1 f_2 \otimes g_1 g_2
    \end{equation}
  for $f_1,f_2,g_1,g_2 \in {\rm sQSym} $ with $g_1$ and $f_2$ of fermionic degree $a$ and $b$ respectively.

 The comultiplication
$\Delta: {\rm sQSym}  \to {\rm sQSym}  \otimes {\rm sQSym}  $ is defined as
\begin{equation}
(\Delta f)(x,y;\theta,\phi) = f(x,y;\theta,\phi)
\end{equation}  
where as we just mentioned, $f(x,y;\theta,\phi)$ is considered an element
of ${\rm sQSym}  \otimes {\rm sQSym}$.   Contrary to the symmetric functions in superspace case, it is not immediately obvious this time that
the coproduct is coassociative given the ordering on the variables
(see the corresponding discussion in \cite{GR} in the non-supersymmetric case).
But we will see in Proposition~\ref{propqsym}
that it easily follows from the next proposition.  

Given the dotted compositions $\ga=(\ga_1,\ldots,\ga_k)$ and
$\gb=(\gb_1,\ldots,\gb_{\ell})$, we define their {\defstyle concatenation} $\ga\cdot\gb$ to be the dotted composition
$(\ga_1,\ldots,\ga_k,\gb_1,\ldots,\gb_{\ell})$.

\begin{proposition}\label{coprod} We have, for any dotted composition $\ga=(\ga_1\dots,\ga_l)$, that
  \begin{equation} \label{coprodM}
    \Delta(M_{\ga})= \sum_{\beta \cdot \gamma=\alpha} M_\beta \otimes M_{\gamma}
=
    \sum_{k=0}^{l} M_{\ga_1, \ldots , \ga_k}\otimes  M_{\ga_{k+1}, \ldots , \ga_l}
    \end{equation}

\end{proposition}
\begin{proof}
  For any $k\in \{ 0,1, \cdots  , l\}$, a monomial in $M_\ga(x,y;\te,\phi)$ is written uniquely in the form $$\te_{i_1}^{\eta_1} \cdots \te_{i_k}^{\eta_k} \phi_{j_1}^{\eta_{k+1}} \cdots \phi_{j_{l-k}}^{\eta_l} x_{i_1}^{\ga_1} \cdots x_{i_k}^{\ga_k} y_{j_1}^{\ga_{k+1}} \cdots y_{j_{l-k}}^{\ga_l} = \te_{i_1}^{\eta_1} \cdots \te_{i_k}^{\eta_k}x_{i_1}^{\ga_1} \cdots x_{i_k}^{\ga_k} \, \phi_{j_1}^{\eta_{k+1}} \cdots \phi_{i_{l-k}}^{\eta_l} y_{j_1}^{\ga_{k+1}} \cdots y_{j_{l-k}}^{\ga_l}$$
with $i_1 < \cdots <i_k$ and $j_1 < \cdots <j_{l-k}$.
\end{proof}

\begin{example}
$$\Delta(M_{\dot{2},1,\dot{3},4})= 1 \otimes M_{\dot{2},1,\dot{3},4}+ M_{\dot{2}}\otimes M_{1,\dot{3},4} +  M_{\dot{2},1}\otimes M_{\dot{3},4} +  M_{\dot{2},1,\dot{3}}\otimes M_{4} +  M_{\dot{2},1,\dot{3},4}\otimes 1$$
\end{example}

\begin{proposition} \label{propqsym}
  The $\mathbb Q$-algebra ${\rm sQSym}$ is a commutative graded connected Hopf algebra.  Moreover, it contains
 the ring of symmetric functions in superspace ${\bf \Lambda}$ as a Hopf subalgebra.
\end{proposition}
\begin{proof}
  The coassociativity of the coproduct is proved by verifying
$(\Delta \otimes {\rm id}) \circ \Delta=({\rm id} \otimes \Delta) \circ \Delta$
  on the monomial basis.  We have
  \begin{align}
\bigl((\Delta \otimes {\rm id}) \circ \Delta\bigr)M_\alpha & = \sum_{k=0}^l \Delta M_{\alpha_1,\dots,\alpha_k} \otimes M_{\alpha_{k+1},\dots,\alpha_l} \nonumber \\ 
& = \sum_{k=0}^l\sum_{i=0}^k M_{\alpha_1,\dots,\alpha_i} \otimes M_{\alpha_{i+1},\dots,\alpha_k} \otimes M_{\alpha_{k+1},\dots,\alpha_l}
  \end{align}
  which shows the coassociativity since $\bigl(({\rm id} \otimes \Delta) \circ \Delta\bigr)M_\alpha$ obviously yields the same result.

  The coproduct is an algebra morphism given that
$(fg)(x,y;\theta,\phi)=f(x,y;\theta,\phi) g(x,y;\theta,\phi)$,
for all $f,g \in {\rm sQSym}$ implies
 \begin{equation}
\Delta(fg)= \Delta(f) \cdot \Delta(g)
\end{equation}
(the other condition in \eqref{morphism} is trivially satisfied).

 The counit $\epsilon$ is as usual the identity on ${\rm sQSym}_{0,0}=\mathbb Q$ and
 the null operator on the rest of ${\rm sQSym}$.  It is easily checked that $\epsilon$ is an algebra morphism.  Defining the grading 
\begin{equation}
{\rm sQSym}_n = \bigoplus_{k+i=n} {\rm sQSym}_{k,i}
\end{equation}  
it is also easy to see that ${\rm sQSym}$ is a graded and connected bialgebra.
By Theorem~\ref{hopftheo}, this
  implies that ${\rm sQSym}$ is a Hopf algebra.
 
Finally,  to prove that ${\bf \Lambda}$ is a Hopf subalgebra
of  ${\rm sQSym}$, we need to prove that
when $\Delta$ is restricted to the subalgebra
 ${\bf \Lambda} \subset {\rm sQSym}$, it is equal to the coproduct $\Delta$ in 
 ${\bf \Lambda}$.  It suffices to prove the claim on $p_i$ and $\tilde p_k$ for
 $i\geq 1 $ and $k\geq 0$ since they generate  ${\bf \Lambda}$.
 Using \eqref{coprode} and Proposition~\ref{coprod}, this is an immediate consequence of the fact that  $p_i=M_{i}$ and $\tilde p_k=M_{\dot k}$.
 
 \end{proof}  

\subsection{Partial orders on compositions} \label{subpartial}

We will define two partial orders on dotted compositions. Given
compositions $\ga$ and $\gb$, we say that $\ga$ covers $\gb$ in the
first partial order, written $\gb \preccurlyeq \ga$, if we can obtain $\ga$
by adding together a pair of adjacent non-dotted parts of $\gb$. The
first partial order is the transitive closure
of this cover relation.  If  $\gb \preccurlyeq \ga$ we say that
{\defstyle $\gb$ strongly refines $\ga$} or that $\ga$ {\defstyle strongly coarsens $\gb$}.

The second partial order on dotted compositions is generated by the
following covering relation: $\gb\trianglelefteq \ga$ if we can
obtain $\ga$ by adding together two adjacent parts of $\gb$, not both
parts dotted (note that adding together a dotted part with an non-dotted one yields a dotted part). If $\gb\trianglelefteq \ga$ we say this time that
{\defstyle $\gb$ weakly refines $\ga$} or that $\ga$ {\defstyle weakly coarsens $\gb$}.

Please see Figure~\ref{Fig:twoPosets} for the two orders.
When no parts are dotted, both covering relations become the covering
 relation on compositions described in \cite{LMW} and \cite{GR}.


\begin{center}
\begin{figure}[ht]
\begin{tikzpicture}[font=\small]
  \def\a{2}
  \def\b{1.5}
  \node (11212) at (0,1*\b) {$(1,1,\dot{2},1,2)$};
  \node (2212) at (-1,2*\b) {$(2,\dot{2},1,2)$};
  \node (1123) at (1,2*\b) {$(1,1,\dot{2},3)$};
  \node (223) at (0,3*\b) {$(2,\dot{2},3)$};

  \draw[thick, red] (11212)--(2212);
  \draw[thick, red] (11212)--(1123);

   \draw[thick, red] (223)--(2212);
   \draw[thick, red] (223)--(1123);

\begin{scope}[shift={(4.5*\a,0)}]

\node (11212) at (0,0*\b) {$(1,1,\dot{2},1,2)$};

\node (2212) at (-1.5*\a,1*\b) {$(2,\dot{2},1,2)$};
\node (1312) at (-.5*\a,1*\b) {$(1,\dot{3},1,2)$};
\node (1132) at (.5*\a,1*\b) {$(1,1,\dot{3},2)$};
\node (1123) at (1.5*\a,1*\b) {$(1,1,\dot{2},3)$};
 
\node (412) at (-2.5*\a,2*\b) {$(\dot{4},1,2)$};
\node (232) at (-1.5*\a,2*\b) {$(2,\dot{3},2)$};
\node (223) at (-.5*\a,2*\b) {$(2,\dot{2},3)$};
\node (115) at (.5*\a,2*\b) {$(1,1,\dot{5})$};
\node (142) at (1.5*\a,2*\b) {$(1,\dot{4},1)$};
\node (133) at (2.5*\a,2*\b) {$(1,\dot{3},3)$};

\node (52) at (-1.5*\a,3*\b) {$(\dot{5},2)$};
\node (43) at (-.5*\a,3*\b) {$(\dot{4},3)$};
\node (25) at (.5*\a,3*\b) {$(2,\dot{5})$};
\node (16) at (1.5*\a,3*\b) {$(1,\dot{6})$};

\node (7) at (0*\a,4*\b) {$(\dot{7})$};

\draw[thick, red] (11212)--(2212);
\draw[thick, red] (11212)--(1312);
\draw[thick, red] (11212)--(1132);
\draw[thick, red] (11212)--(1123);

\draw[thick, red] (412)--(2212);
\draw[thick, red] (232)--(2212);
\draw[thick, red] (223)--(2212);
\draw[thick, red] (412)--(1312);
\draw[thick, red] (142)--(1312);
\draw[thick, red] (133)--(1312);
\draw[thick, red] (232)--(1132);
\draw[thick, red] (115)--(1132);
\draw[thick, red] (142)--(1132);
\draw[thick, red] (223)--(1123);
\draw[thick, red] (133)--(1123);

\draw[thick, red] (412)--(52);
\draw[thick, red] (232)--(52);
\draw[thick, red] (223)--(43);
\draw[thick, red] (412)--(43);
\draw[thick, red] (142)--(52);
\draw[thick, red] (133)--(43);
\draw[thick, red] (232)--(25);
\draw[thick, red] (115)--(25);
\draw[thick, red] (115)--(16);
\draw[thick, red] (142)--(16);
\draw[thick, red] (223)--(25);
\draw[thick, red] (133)--(16);

\draw[thick, red] (7)--(52);
\draw[thick, red] (7)--(43);
\draw[thick, red] (7)--(25);
\draw[thick, red] (7)--(16);

\end{scope}

  \end{tikzpicture}
  \caption{The poset on the left is all dotted compositions above
  $(1,1,\dot{2},1,2)$ using the partial partial order ($\preccurlyeq$). On the
  right the poset is again all dotted compositions above $(1,1,\dot{2},1,2)$,
  but using the partial partial order $(\trianglelefteq)$.}
\label{Fig:twoPosets}
\end{figure}
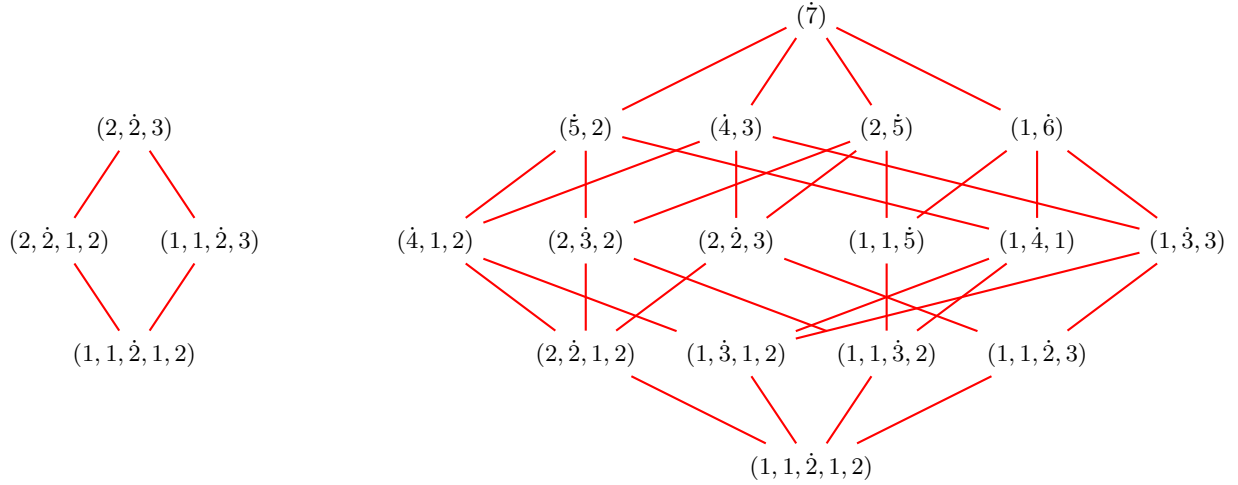
\end{center}

\subsection{Antipode}
 We now give the action of the antipode $S: {\rm sQSym} \to {\rm sQSym}$ 
  explicitly on monomial quasisymmetric
  functions in superspace.

Let the {\defstyle reverse} of a composition $\ga=(\ga_1,\ldots,\ga_k)$ be
$\Rev(\ga)=(\ga_k,\ldots,\ga_1)$.
 \begin{proposition}
\label{prop:antiM}
  Let $\ga$ be a dotted composition. Then
\begin{equation}
  S(M_{\ga})=(-1)^{\ell(\ga) + \binom{m_\ga}{2} }\sum_{\gamma\,  \trianglerighteq \,  \Rev(\ga)}M_{\gamma}
\end{equation}
where $m_\alpha$ is the fermionic degree of the dotted composition $\alpha$.
  \end{proposition}
Before proving the proposition, we give a few examples:
\begin{example}
$$\quad S(M_{\dot{1}, 3, \dot{2}}) = (-1)^{3 + \binom{2}{2}}
\left ( M_{\dot{2},3 ,\dot{1}} + M_{\dot{5},\dot{1}} +
M_{\dot{2} ,\dot{4}} \right)$$
and
\begin{align*}
S(M_{\dot{3}, 2, \dot{2}, 1,\dot{1}}) & = (-1)^{5 +
\binom{3}{2}} \Bigl( M_{\dot{1},1,\dot{2},2 ,\dot{3}} +
M_{\dot{2},\dot{2},2, \dot{3}} + M_{\dot{2},\dot{2} ,\dot{5}} +
M_{\dot{2},\dot{4}, \dot{3}} + M_{\dot{1},1,\dot{2}, \dot{5}} \, + \Bigr. \\
& \qquad \qquad \qquad \qquad   \qquad   \qquad 
\Bigl.  M_{\dot{1},\dot{3} ,2,\dot{3}} + M_{\dot{1},\dot{3},
\dot{5}} + M_{\dot{1},\dot{5}, \dot{3}} + M_{\dot{1},1,\dot{4}, \dot{3}} +
M_{\dot{1},1,\dot{2}, \dot{5}} \Bigr )
\end{align*}
\end{example}

\begin{proof}
  The proof proceeds by induction on $\ell=\ell(\alpha)$.  It 
  generalizes that of \cite{E} (also given in \cite{GR})
  in the quasisymmetric case.

 We prove the base cases $\ell=0,1$ directly. 
  For $\ell=0$, we have $S(M_{\emptyset})=S(1)=(-1)^0 M_{\Rev(\emptyset)}$.
  For $\ell=1$, we have from Proposition~\ref{coprod} that
  $M_{r}$ and $M_{\dot r}$ are primitive elements.  Therefore,
  $S(M_{r}) =-M_{r}= (-1)^{1+\binom{0}{2}} M_{r}$ and $S(M_{\dot r}) =-M_{\dot r}= (-1)^{1+\binom{1}{2}} M_{\dot r}$ and the result holds for $\ell=1$.

  For $\ell(\alpha)\geq 2$,   we need to verify, by
 \eqref{eqantipo} and \eqref{coprodM},   that 
\begin{equation}
  S(M_{\alpha_1, \ldots, \alpha_\ell})  =   - \sum_{i=0}^{\ell-1} S(M_{\alpha_1, \ldots, \alpha_i} ) \cdot M_{\alpha_{i+1}, \ldots, \alpha_\ell}
\end{equation}
holds. By induction, this amounts to checking the following identity:
 \begin{equation} \label{eqanti1}
   (-1)^{\ell(\ga) + \binom{m_\ga}{2} }\sum_{\gamma\,  \trianglerighteq \,  \Rev(\ga)}M_{\gamma}
   =  \sum_{i=0}^{\ell-1} \sum_{\tiny{\beta  \trianglerighteq \alpha_i,\ldots , \alpha_1}}(-1)^{i+1+ \binom{m_\beta}{2}}  M_\beta \cdot M_{\alpha_{i+1}, \ldots , \alpha_\ell}
\end{equation}
 We will see that most terms cancel two by two in the expansions of the products
 $M_\beta \cdot M_{\alpha_{i+1}, \ldots, \alpha_\ell}$
 in \eqref{eqanti1}, and that those that do not cancel are exactly
 the $M_\gamma$'s such that $\gamma \trianglerighteq \Rev(\alpha)$ (with the right sign).

 Unless $\beta=\emptyset$, the first part of $\beta$ is of the form
 $\beta_1= \ga_i+\ga_{i-1}+\cdots \ga_h$ where $h\leq i$.  Hence each term $M_\gamma$ in the expansion of  $M_\beta \cdot M_{(\alpha_{i+1}, \ldots, \alpha_\ell)}$ is such that its first entry $\gamma_1$ has one of the possible three forms:
\begin{enumerate}
\item [I.]  $\gamma_1 = \ga_i+\ga_{i-1}+\cdots +\ga_h$
\item[II.] $\gamma_1 = \ga_{i+1} + (\ga_i+\ga_{i-1}+\cdots +\ga_h)$
\item[III.] $\gamma_1=\ga_{i+1}$
\end{enumerate}
We will see that, for $i=1,\dots,\ell-1$, the
terms of type I in the case $i$ cancel with those of type II in the
case $i-1$, and that similarly, the terms of type I in the case $i$ cancel with those of type III in the case $i-1$.

Suppose that $\beta$ and $\beta'$ are such that their
only difference occurs in the first entry: $\beta_1
= \ga_{i}+\ga_{i-1}+\cdots +\ga_h$ (type I in case $i$)
while $\beta_1' = \ga_{i-1}+\cdots + \ga_h$ (type II in case $i-1$).
Since $\beta_1=\beta_1'+\alpha_i$, 
the two paths $P_\gamma$ and $P_\gamma'$ in  Figure~\ref{casoIyII}
(representing type I and II respectively) produce the same dotted composition with signs given by
$(-1)^{i+1 + \binom{m_\beta}{2}}
\sign(P_\gamma)$ and $(-1)^{(i-1)+1 + \binom{m_\beta'}{2}}
\sign(P'_\gamma)$ respectively. If  $\alpha_i$ is not dotted, then $m_\beta =
m_\beta'$ and $\sign(P'_\gamma)= \sign(P_\gamma)$ which means that
the terms have opposite
signs.
Otherwise, $\alpha_i$ is dotted which implies that
$m_{\beta'} = m_\beta-1$ and
$\sign(P'_\gamma)= (-1)^{m_\beta -1}\sign(P_\gamma)$ since there are $m_\beta$ extra dots below the path $P_\gamma'$.  Using the fact that
$\binom{m-1}{2}+m-1=\binom{m}{2}$, wet get again that the terms have opposite
signs.

\begin{center}
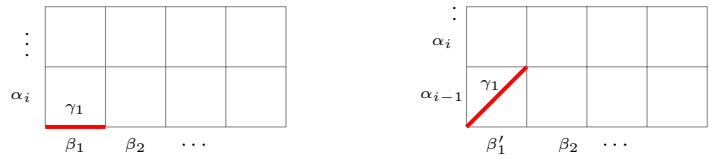
\begin{figure}[ht] 
  \begin{tikzpicture}[fill opacity=1, scale=0.8,font=\footnotesize]
\def\r{3.5}
\def\b{3.5}
\def\con{3.8}
\def\co{8.4}
\def\coa{8.5}
\def\ct{23}
\def\cta{24.5}
\def\labelHt{-.45}
\def\permHt{1}
\def\smallpic{4}
\def\bigpic{5}
\def\permx{3.8}

\begin{scope}[shift={(0,0*\r)}]
\node at (.5,-0.3){\tiny $\beta_1$}; \node at (1.5,-0.3){\tiny
$\beta_2$}; \node at (2.5,-0.3){$\cdots$}; \node at (-.4,.5){\tiny
$\alpha_{i}$}; \node at (.5,.3){\tiny $\gamma_1$}; \node at
(-.3,1.5){$\vdots$}; \draw[help lines] (0,0) grid (4,2);

\draw[red,line width=1.5pt](0,0)-- ++(1,0);

\end{scope}

\begin{scope}[shift={(7,0*\r)}]
\node at (0.5,-0.3){\tiny $\beta_1'$}; \node at (1.7,-0.3){\tiny
$\beta_2$}; \node at (2.5,-0.3){$\cdots$}; \node at (-0.4,.5){\tiny
$\alpha_{i-1}$}; \node at (-0.4,1.4){\tiny $\alpha_{i}$}; \node at
(-0.2,2.1){$\vdots$}; \draw[help lines] (0,0) grid (4,2);

\node at (.4,.7){\tiny $\gamma_1$};
 \draw[red,line width=1.5pt](0,0)-- ++(1,1);

\end{scope}

\end{tikzpicture}
  \caption{Paths $P_\gamma$ and $P_\gamma'$ of type I and II}
  \label{casoIyII}
  \end{figure}
\end{center}

For the other case, suppose that
$\beta=(\beta_1, \cdots , \beta_k)$  (type I in case $i+1$) and $\beta'=(\alpha_{i+1},
\beta_1, \cdots , \beta_k)$  (type III in case $i$).  Then the two paths $P_\gamma$ and $P_\gamma'$ in  Figure~\ref{casoIIIyI}
(representing type I and II respectively) produce the same dotted composition with signs given by
 $(-1)^{i+1 \binom{m_\beta}{2}} \sign(P_\gamma)$ and $(-1)^{i+1+1
 \binom{m_\beta'}{2}} \sign(P_\gamma')$ respectively. If $\alpha_{i+1}$ is not
 dotted then $m_\beta' = m_\beta$ and $\sign(P'_\gamma)=
 \sign(P_\gamma)$ which means that the terms have opposite signs.
Otherwise, $\alpha_{i+1}$ is dotted which means that
 $m_\beta' = m_\beta +1$ and $\sign(P'_\gamma)=
\sign(P_\gamma) (-1)^{-m_\beta}$ since $P_\gamma$ has $m_\beta$ extra dots.
The terms are again seen to have opposite signs.
\begin{center}
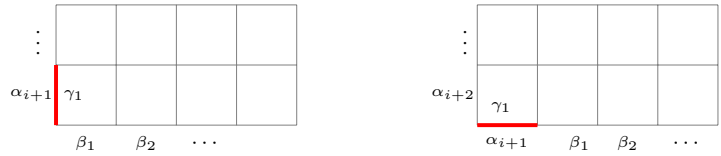
\begin{figure}[ht]
  \begin{tikzpicture}[fill opacity=1, scale=0.8,font=\footnotesize]

\def\r{3.5}

\begin{scope}[shift={(0,0*\r)}]
\node at (.5,-0.3){\tiny $\beta_1$}; \node at (1.5,-0.3){\tiny
$\beta_2$}; \node at (2.5,-0.3){$\cdots$}; \node at (-0.4,.5){\tiny
$\alpha_{i+1}$}; \node at (.3,.5){\tiny $\gamma_1$}; \node at
(-.3,1.5){$\vdots$}; \draw[help lines] (0,0) grid (4,2);

\draw[red,line width=1.5pt](0,0)-- ++(0,1);

\end{scope}

\begin{scope}[shift={(7,0*\r)}]
\node at (0.5,-0.3){\tiny $\alpha_{i+1}$}; \node at (1.7,-0.3){\tiny
$\beta_1$};  \node at (2.5,-0.3){\tiny $\beta_2$};
 \node at (3.5,-0.3){$\cdots$}; \node at (-0.4,.5){\tiny
$\alpha_{i+2}$}; \node at (-0.2,1.5){$\vdots$}; \draw[help lines]
(0,0) grid (4,2);

\node at (.4,.3){\tiny $\gamma_1$};
 \draw[red,line width=1.5pt](0,0)-- ++(1,0);

\end{scope}

\end{tikzpicture}
  \caption{Paths $P_\gamma$ and $P_\gamma'$ of type I and III}
  \label{casoIIIyI}
  \end{figure}
\end{center}
We are left with the case $i=\ell-1$ for type II and III (see Figure~\ref{casoIIIyIb}) which corresponds to
$(-1)^{\ell+\binom{m_\alpha}{2}}M_\beta M_{(\alpha_\ell)}$, with $\beta \trianglerighteq (\alpha_{\ell-1},\ldots,\alpha_1)$.   It is easy to see that those are exactly the $\gamma$'s such that  $\gamma \trianglerighteq \Rev(\alpha)$.
\begin{center}
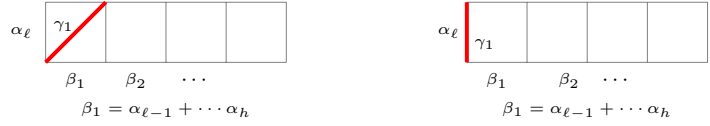
\begin{figure}[ht]
  \begin{tikzpicture}[fill opacity=1, scale=0.8,font=\footnotesize]

\def\r{3.5}

\begin{scope}[shift={(0,0*\r)}]
\node at (.5,-0.3){\tiny $\beta_1$}; \node at (1.5,-0.3){\tiny
$\beta_2$}; \node at (2.5,-0.3){$\cdots$}; \node at (-.4,.5){\tiny
$\alpha_{\ell}$}; \node at (.3,.6){\tiny $\gamma_1$};  \draw[help
lines] (0,0) grid (4,1);

\node at (2,-0.8){\tiny $\beta_1=\alpha_{\ell-1}+\cdots \alpha_h$};

 \draw[red,line width=1.5pt](0,0)-- ++(1,1);

\end{scope}

\begin{scope}[shift={(7,0*\r)}]
\node at (0.5,-0.3){\tiny $\beta_1$}; \node at (1.7,-0.3){\tiny
$\beta_2$};
 \node at (2.5,-0.3){$\cdots$}; \node at (-0.3,.5){\tiny
$\alpha_\ell$}; \draw[help lines] (0,0) grid (4,1);

\node at (2,-0.8){\tiny $\beta_1=\alpha_{\ell-1}+\cdots \alpha_h$};

\node at (.3,.3){\tiny $\gamma_1$};
 \draw[red,line width=1.5pt](0,0)-- ++(0,1);

\end{scope}

\end{tikzpicture}
  \caption{Paths of type II and III}
  \label{casoIIIyIb}
  \end{figure}
\end{center}

\end{proof}

\subsection{Fundamental quasisymmetric functions in superspace} \label{Secfund}
The fundamental quasisymmetric functions provide a very important basis of
${\rm QSym}$ whose properties are reminiscent of those of Schur functions
\cite{GR,LMW, S}.
It is thus natural to look for their generalization to superspace.
But just as there are two natural extensions to superspace of Schur functions, $s_\Lambda$ and $\bar s_\Lambda$, there are two natural candidates to generalize the fundamental quasisymmetric functions.  They are defined using the 
partial orders introduced in Section~\ref{subpartial}.
\begin{definition}
  \label{def:fundamental}
The {\defstyle fundamental quasisymmetric function in superspace} $\funda$ is
\begin{equation}
  \funda=\sum_{\gb  \preccurlyeq \ga}M_{\gb}
  \end{equation}
while the {\defstyle fundamental quasisymmetric function in superspace} $\bar L_\alpha$ is
\begin{equation}
\bar L_\alpha= \sum_{\gb \trianglelefteq  \ga}M_{\gb}
\end{equation}
\end{definition}
\begin{example}
  We have
  $$
  L_{3,\dot{4},2} = M_{3,\dot{4},2}+  M_{2,1,\dot{4},2}+
M_{1,2,\dot{4},2}+  M_{1,1,1,\dot{4},2}+  M_{3,\dot{4},1,1}+
M_{2,1,\dot{4},1,1}+  M_{1,2,\dot{4},1,1}+  M_{1,1,1,\dot{4},1,1}
  $$
 and
\begin{align}
\bar{L}_{\dot{2},2} &= M_{\dot{2},2}+ M_{\dot{0},2,2}+
M_{2,\dot{0},2} + M_{\dot{1},1,2}+ M_{1,\dot{1},2}+
M_{\dot{0},1,1,2}+ M_{1,\dot{0},1,2}+ M_{1,1,\dot{0},2}   \nonumber \\
 &+ M_{\dot{2},1,1}+ M_{\dot{0},2,1,1}+ M_{2,\dot{0},1,1} +
M_{\dot{1},1,1,1}+ M_{1,\dot{1},1,1}+ M_{\dot{0},1,1,1,1}+
M_{1,\dot{0},1,1,1}  
+ M_{1,1,\dot{0},1,1} \nonumber
\end{align}
\end{example}
The study of the functions $L_\alpha$ and $\bar L_\alpha$ turns out to
be somewhat more involved than in the usual quasisymmetric case.
For instance, the action of the antipode $S$ on $L_\alpha$ and $\bar L_\alpha$
is not trivial while in ${\rm QSym}$ the antipode $S$ acting
on $L_\alpha$ is simply the element $L_{\alpha^t}$ (up to a sign), where $\alpha^t$ is the composition whose corresponding ribbon is the transposed of that of $\alpha$.
In fact, $S(L_\alpha)$ and $S(\bar L_{\alpha})$ are interesting bases in their own right just as one could say that their counterparts in ${\mathbf \Lambda}$,
$\omega s_\Lambda$ and $\omega \bar s_\Lambda$, are (except that in the latter case those bases are, from Proposition~\ref{propdual}, essentially dual to the
$s_\Lambda$ and $\bar s_\Lambda$ bases).
Because of the intricacies of the combinatorics at play, we will study the fundamental quasisymmetric functions in superspace in a forthcoming article \cite{FLP}
where we will see for instance that the products of $L_\alpha$'s or $\bar L_\alpha$'s are described using new types of shuffles (called weak and strong respectively), and that the Schur functions in superspace $s_\Lambda$ and $\bar s_\Lambda$
expand naturally in terms of the $L_\alpha$'s and $\bar L_\alpha$'s respectively.

\section{Noncommutative symmetric functions in superspace}

The Hopf algebra  ${\rm NSym}$ of noncommutative symmetric functions is dual to 
${\rm QSym}$.  This duality can be extended to superspace given that ${\rm sQSym}$ is graded and that each of its homogeneous component is finite dimensional.
In the following, we generalize to superspace the presentation of \cite{GR}.
\begin{definition} Let  ${\rm sNSym}$  be the Hopf dual of ${\rm sQSym}$
with dual pairing  $\langle \cdot , \cdot \rangle : {\rm sNSym} \otimes {\rm sQSym} \to \mathbb Q$.  The Hopf algebra ${\rm sNSym}$  has a  $\mathbb Q$-basis  $\{H_\alpha \}$  dual to the monomial quasisymmetric functions in superspace, that is, such that
  \begin{equation}
   \langle H_\alpha, M_\beta \rangle = \delta_{\alpha \beta}
  \end{equation}  
  \end{definition}  
We call ${\rm sNSym}$ the ring of {\defstyle noncommutative symmetric functions in superspace}.
\begin{proposition}  Let $H_m= H_{(m)}$ and $\tilde H_n=H_{(\dot n)}$
  for $m\geq 1$ and $n \geq 0$.  We have that
  \begin{equation} \label{eqcong}
   {\rm sNSym} \cong \mathbb Q\langle H_1,H_2,\dots; \tilde H_0,\tilde H_1,\dots \rangle
  \end{equation}
  the free associative algebra with noncommuting generators $\{  H_1,H_2,\dots; \tilde H_0,\tilde H_1,\dots \}$ and coproduct defined by
  \begin{equation} \label{eqNcop}
    \Delta H_n  = \sum_{i+j=n } H_i \otimes H_j \qquad {\rm and} \qquad
\Delta \tilde H_i = \sum_{k+\ell=i} 
 \left( \tilde H_k \otimes H_\ell + H_\ell \otimes \tilde H_k \right)
  \end{equation}  
\end{proposition}  
\begin{proof}
  Using the coproduct $\Delta M_\alpha= \sum_{\beta \cdot\gamma=\alpha} M_\beta \otimes M_\gamma$, we have by duality that
  \begin{equation}
    H_\beta H_\gamma = H_{\beta\cdot\gamma}
  \end{equation}  
  This readily implies that $H_\alpha=H_{\alpha_1} \cdots H_{\alpha_l}$ where
  $H_{\dot r}:=\tilde H_r$, which proves \eqref{eqcong}.

  Since $H_m$ and $\tilde H_{n}$ are dual to $M_{m}$ and $M_{\dot n}$ respectively,
  we only need to know which products $M_\alpha M_\beta$ can generate a term of the form $M_m$ or $M_{\dot n}$.
From Proposition~\ref{propmultM}, we have
  \begin{equation}
   M_a  M_b  = M_{a+b} + M_{a,b}+M_{b,a}\,, \quad 
   M_{\dot a} M_{b}= M_{\dot{(a+b)}}+M_{\dot a,b}+M_{b, \dot a} \quad
   {\rm and } \quad M_{a} M_{\dot b}= M_{\dot{(a+b)}}+M_{a,\dot b}+M_{\dot b, a}
  \end{equation}  
By duality, \eqref{eqNcop} holds.
  \end{proof}
\begin{corollary} The algebra morphism $\pi: {\rm sNSym} \to {\bf \Lambda}$ defined by
  \begin{equation}
    \pi(H_m)=h_m \qquad {\rm and} \qquad \pi(\tilde H_n)=\tilde h_n
  \end{equation}  
  is a Hopf algebra surjection such that
  \begin{equation} \label{eqduali}
     \langle \! \langle \, \pi(F), g \,  \rangle \! \rangle = \langle F, \iota(g) \rangle 
  \end{equation}
  where $\iota: {\bf \Lambda} \to  {\rm sQSym}$ is the inclusion map, and where we are using our usual scalar product on ${\bf \Lambda}$. The relationship between ${\bf \Lambda},  {\rm sQSym}$ and $ {\rm sNSym}$
is  illustrated in Figure~\ref{figtri}.
\begin{figure}
\begin{tikzpicture}  
\def\a{4}
\def\b{4}
\node[ellipse,fill=blue!40] (nsym) at (-\a,\b) {sNSym};
\node[ellipse,fill=blue!40] (qsym) at (\a,\b) {sQSym};
\node[ellipse,fill=blue!40] (sym) at (-\a*0,\b*0) {${\bf \Lambda}$};
\draw[->>,ultra thick, red]  (nsym)--(sym)node[midway,left,blue]{$\pi$};
\draw[<->,ultra thick, red,out=45,in=135](nsym)--(qsym) node[midway,above, blue] {$\langle H_{\alpha},M_{\beta}\rangle=\delta_{\alpha\beta}$};
\draw[right hook->,ultra thick, red](sym)--(qsym) node[midway,right, blue]{$\iota$};
\end{tikzpicture}
\caption{Relationship between ${\bf \Lambda},  {\rm sQSym}$ and $ {\rm sNSym}$} \label{figtri}
\end{figure}
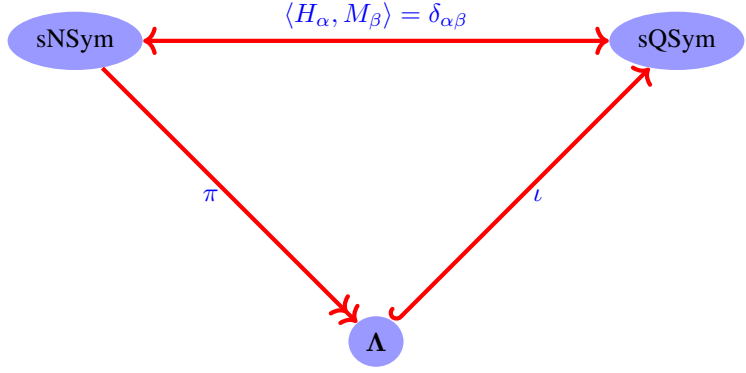
\end{corollary} \label{corofinal}
\begin{proof}  Since $\Lambda$ is generated by $h_1,h_2,\dots; \tilde h_0,\tilde h_1,\dots$, the map $\pi$ is automatically surjective.  Comparing
  \eqref{eqcoprodh} and \eqref{eqNcop}, we get that $\pi$ is also
  a coalgebra morphism.  The map $\pi$ is then a Hopf algebra morphism (that is, it respects the antipode) since any bialgebra morphism is a Hopf algebra morphism \cite{GR}.

  We will prove \eqref{eqduali} by showing that it holds on 
the $\{H_\alpha \}$ and $\{m_\Lambda\}$ basis.  Using \eqref{dualhm} and the notation of \eqref{menM}, this is indeed the case since
  \begin{equation}
 \langle \! \langle \, \pi(H_\alpha), m_\Lambda \,  \rangle \! \rangle = (-1)^{\sigma(\alpha)}\delta_{\tilde \alpha \gamma} =  \langle H_{\alpha},  \sum_{\alpha \, : \, \tilde \beta=\gamma} (-1)^{\sigma(\beta)} M_\beta \rangle
   \end{equation} 
where we recall that $\gamma=(\dot \Lambda^a_1, \dots, \dot \Lambda^a_m,\Lambda^s_1,\dots,\Lambda^s_l)$ is $\Lambda$ considered as a dotted composition. 
\end{proof}  

\begin{acknow} We would like to thank one of the reviewers for pointing out that a sign was missing in the anti-homomorphism relation satisfied by the antipode.
 S.F. thanks the Universidad de Talca for its warm hospitality during three extended stays. 
  This work was supported in part by
 the Simons Foundation (grant \#359602, S.F.) and by
FONDECYT (Fondo Nacional de Desarrollo Cient\'{\i}fico y Tecnol\'ogico de Chile) through the  {initiation grant \#{11140280}} (M.E. P.) and regular grant \#{1170924} (L. L.).
\end{acknow}

\end{document}